\documentclass[openany, amssymb, psamsfonts]{amsart}
\usepackage{geometry}
\usepackage{mathrsfs,comment}
\usepackage[usenames,dvipsnames]{color}
\usepackage[normalem]{ulem}
\usepackage{url}
\usepackage[all,arc,2cell]{xy}
\UseAllTwocells
\usepackage{enumerate}

\usepackage[utf8]{inputenc}
\usepackage{amsmath} 
\usepackage{amssymb}
\usepackage{amsthm} 
\usepackage{mathtools}
\usepackage[shortlabels]{enumitem}
\usepackage[dvipsnames]{xcolor}
\usepackage{tensor} 
\usepackage{comment} 
\usepackage{relsize}
\usepackage[english]{babel}
\usepackage{mathrsfs}
\usepackage{xypic} 
\usepackage{hyperref}
\usepackage{physics} 
\usepackage{bbm} 
\usepackage{csquotes}
\usepackage[
backend=biber,
style=alphabetic
]{biblatex}
\makeatletter
\@ifpackageloaded{txfonts}\@tempswafalse\@tempswatrue
\if@tempswa
  \DeclareFontFamily{U}{txsymbols}{}
  \DeclareFontFamily{U}{txAMSb}{}
  \DeclareSymbolFont{txsymbols}{OMS}{txsy}{m}{n}
  \SetSymbolFont{txsymbols}{bold}{OMS}{txsy}{bx}{n}
  \DeclareFontSubstitution{OMS}{txsy}{m}{n}
  \DeclareSymbolFont{txAMSb}{U}{txsyb}{m}{n}
  \SetSymbolFont{txAMSb}{bold}{U}{txsyb}{bx}{n}
  \DeclareFontSubstitution{U}{txsyb}{m}{n}
  \DeclareMathSymbol{\aleph}{\mathord}{txsymbols}{64}
  \DeclareMathSymbol{\beth}{\mathord}{txAMSb}{105}
  \DeclareMathSymbol{\gimel}{\mathord}{txAMSb}{106}
  \DeclareMathSymbol{\daleth}{\mathord}{txAMSb}{107}
\fi
\makeatother

\newcommand{\floor}[1]{\lfloor #1 \rfloor}

\newcommand{\Lie}[1]{\mathrm{Lie}\left(#1\right)}

\newcommand{\supp}[1]{\mathrm{supp}\left(#1\right)}
\newcommand{\aA}{{{\scriptstyle{}^{\scriptstyle a}}\!\!\mathcal{A}}}

\DeclareMathOperator{\C}{\mathbb{C}}
\DeclareMathOperator{\R}{\mathbb{R}}

\DeclareMathOperator{\Z}{\mathbb{Z}}
\DeclareMathOperator{\N}{\mathbb{N}}

\DeclareMathOperator{\F}{\mathbb{F}}
\DeclareMathOperator{\T}{\mathbb{T}}
\DeclareMathOperator{\bL}{\mathbf{L}}

\DeclareMathOperator{\id}{id}
\DeclareMathOperator{\ad}{ad}
\DeclareMathOperator{\Ad}{Ad}
\DeclareMathOperator{\Aut}{Aut}
\DeclareMathOperator{\st}{\text{ s.t.\ }}

\DeclareMathOperator{\divides}{ \, \big|}

\DeclareMathOperator{\foreach}{\mathrm{ \ \forall }}

\DeclareMathOperator{\inv}{^{-1}}

\newtheorem{theorem}{Theorem}[section]
\newtheorem{proposition}[theorem]{Proposition}

\newtheorem{lemma}[theorem]{Lemma}
\newtheorem{remark}[theorem]{Remark}

\theoremstyle{definition}
\newtheorem{definition}{Definition}[section]

\title{Topological Properties of Almost Abelian Lie Groups}

\author[Z. Avetisyan ]{Zhirayr Avetisyan}
\author[O. Buran]{Oderico-Benjamin Buran}
\author[A. Paul]{Andrew Paul}
\author[L. Reed]{Lisa Reed}

\begin{document}

\begin{abstract}
An almost Abelian Lie group is a non-Abelian Lie group with a codimension 1 Abelian subgroup. We show that all discrete subgroups of complex simply connected almost Abelian groups are finitely generated. The topology of connected almost Abelian Lie groups is studied by expressing each connected almost Abelian Lie group as a quotient of its universal covering group by a discrete normal subgroup. We then prove that no complex connected almost Abelian group is compact, and give conditions for the compactness of connected subgroups of such groups. Towards studying the homotopy type of complex connected almost Abelian groups, we investigate the maximal compact subgroups of such groups.
\end{abstract}

\maketitle


\section{Introduction}

Almost Abelian Lie groups are prevalent in math and nature. Most Bianchi groups (those having Lie algebras Bi(II)-Bi(VII)), and therefore many cosmological models, are almost Abelian. Other applications include integrable systems, PDEs, and linear dynamical systems. Of particular import is the fact that the three dimensional Heisenberg group is almost Abelian. As the aforementioned Bianchi and Heisenberg groups indicate, almost Abelian groups include some of the most computationally friendly Lie groups. 

Another area in pure mathematics where almost Abelian Lie groups appear is in the study of solvmanifolds. A solvmanifold is a quotient $G/H$ of a simply connected solvable Lie group $G$ and a discrete subgroup $H$. Almost Abelian solvmanifolds have seen extensive study in recent years, and complex almost Abelian solvmanifolds have been of particular interest; see: \cite{Andriot_2011}, \cite{Andrada_2017}, \cite{freibert2011cocalibrated}, \cite{fino2022balanced}, \cite{Fino_2021}, and \cite{Stanfield_2021}.

General properties of almost Abelian Lie algebras over arbitrary fields were studied in \cite{Ave16} and \cite{Ave18}. Meanwhile, general properties of real almost Abelian Lie groups were studied in \cite{AAB20}. We now provide the complex analogue: we analyze various important structures of complex almost Abelian groups. Indeed, the observant reader will notice that some (but not all) of our proofs and results mirror those in \cite{AAB20}. We try to give explicit descriptions for as many important (topological and algebraic) structures as we can. 

The main result of this paper is Theorem \ref{AllDiscSubgrpsOfSimpConnGrpAreFinGen}, where we prove that every discrete subgroup of a complex connected almost Abelian Lie group is finitely generated.

Recall (as we will also go over in Section \ref{PreliminariesSection}) a \textit{multiplicity function} $\aleph$ completely determines an almost Abelian group via the Jordan matrix $J(\aleph)$ and its Lie algebra, which we denote by $\aA(\aleph)$ (also see \cite{Ave18}).

In Proposition \ref{SimpConnRep} we find that one faithful matrix representation for a simply connected almost Abelian Lie group $G$ with multiplicity function $\aleph$ is given by:
\begin{equation*}
    G \coloneqq \left\{ 
    \begin{pmatrix}
    1 & 0 & 0 \\
    v & e^{tJ(\aleph)} & 0 \\
    t & 0 & 1
    \end{pmatrix} 
    \: \middle| \ (v,t) \in \C^d \oplus \C \right\},
\end{equation*}
and in Lemma \ref{SimpConnExp} we calculate that the exponential map $\exp:\aA(\aleph) \longrightarrow G$ for a simply connected almost Abelian group $G$ with the above representation.

In Proposition \ref{CenterSimpConn} we give a complete description of the center of a simply connected almost Abelian group $G$:
\begin{align*}
    Z(G) = \{(u,s) \in \C^d \rtimes \C \divides u \in \ker(J(\aleph)),\ e^{s J(\aleph)} = \mathbbm{1} \}.
\end{align*}

In Proposition \ref{DiscSubgrpDim} we find that every discrete normal subgroup $N \subseteq G$ of a simply connected almost Abelian group $G$ with Lie algebra $\aA(\aleph)$ is free and finitely generated, and the rank $k$ of a discrete normal subgroup is bounded:
\begin{equation*}
k \leq \dim_{\R}\big(\ker(J(\aleph))\big) + 2.
\end{equation*}
As mentioned earlier, we further prove in Theorem \ref{AllDiscSubgrpsOfSimpConnGrpAreFinGen} that every (not necessarily normal) discrete subgroup of a simply connected almost Abelian group is also finitely generated.

In Proposition \ref{ConnSubGrpForms}, we give the explicit form of all connected Lie subgroups of a simply connected almost Abelian group $G$.

In Proposition \ref{NonCompactConnAAGrp} we find that there are no compact connected almost Abelian groups, and give a necessary and sufficient condition (Proposition \ref{CompCondition}) for a connected Lie subgroup of a connected almost Abelian group to be compact.

Lastly, in section \ref{SectionHomogeneousSpaces} we lay the groundwork for future investigations into homogeneous spaces by proving that the intersection of complex connected Lie subgroups of a simply connected almost Abelian group is again a complex connected Lie subgroup (Lemma \ref{intersectComplexSubgrps}), and find that the maximal compact subgroup of a connected almost Abelian group $G \coloneqq \widetilde{G} / \Gamma$ (where $\widetilde{G}$ is the universal cover and $\Gamma$ is a discrete subgroup) is exactly $\mathcal{C}(\Gamma) / \Gamma$, where $\mathcal{C}(\Gamma)$ is the minimal connected complex Lie subgroup of $\widetilde{G}$ containing $\Gamma$ (Proposition \ref{MaxCompSubgrp}).


\section{Preliminaries}\label{PreliminariesSection}

An almost Abelian Lie algebra is a Lie algebra with a codimension 1 Abelian subalgebra. For a finite-dimensional almost Abelian Lie algebra, this data can be fully captured by a formal device known as an $\N$-graded multiplicity function, and we now summarize from \cite{Ave18} this correspondence.

Let $\mathcal{C}$ be the class of cardinals, and let $\F$ be a field. An \textit{$\N$-graded multiplicity function} $\aleph$ is a map $\aleph : \F \times \N \to \mathcal{C}$. For our purposes, we take $\aleph: \C \times \N \to \mathcal{C}$. It is known (Prop. 1 in \cite{Ave16}) that an almost Abelian Lie algebra is necessarily of the form $\mathbf{V} \rtimes_{\ad_{e_0}} \C e_0$. A multiplicity function $\aleph$ completely and uniquely determines the structure of a complex almost Abelian Lie algebra by determining a Jordan matrix $J(\aleph)$ that serves as a matrix representation for $\ad_{e_0}$. We now give details to how $J(\aleph)$ is defined.

\begin{definition}
Define $\supp{\aleph} \coloneqq \{p \in \C[X] \divides \exists n \in \N \st \aleph(p,n) \neq 0\}$.
\end{definition}

Then, we define $J(p,n) = \lambda_p \mathbbm{1} + N_n$, where $\lambda_p$ is the complex number identified with the monic irreducible polynomial $p \in \C[x]$ that has it as a root, and where $N_n$ is the $n \times n$ matrix with 1's on the superdiagonal and zeroes everywhere else. Then 
\begin{equation*}
    J(\aleph) \coloneqq \bigoplus_{p \in \supp{\aleph}} \bigoplus_{n=1}^\infty \bigoplus_{\aleph(p,n)} J(p,n).
\end{equation*}

For the entirety of this paper, we only consider (finite-dimensional) complex almost Abelian Lie groups with a (finite-dimensional) complex almost Abelian Lie algebra, which then corresponds to a finitely-supported multiplicity function $\aleph$. We represent the Lie algebra uniquely determined by $\aleph$ as:
\begin{definition}
We define $ \aA(\aleph) \coloneqq \aA_{\C}(\aleph) \coloneqq \mathbf{V} \rtimes_{\ad_{e_0}} \C e_0$ where $\text{ad}_{e_0} = J(\aleph),\ \mathbf{V} = \C^{\dim_{\C}(\aleph)}$.
\end{definition}


Using $\aleph$, we can also define the following sets.

\begin{definition} \label{TAlephDef}
We define $T_\aleph \coloneqq \left\{z \in \C | \; e^{zJ(\aleph)} = \mathbbm{1} \right\}$.
\end{definition}

\begin{definition} \label{XAlephDef}
We define $\mathcal{X}_{\aleph}\coloneqq\left\{\omega\in\mathbb{C}\colon\supp{\aleph}\cong S\subseteq i \omega\mathbb{Z}\right\}$.
\end{definition}

\begin{lemma} \label{TalephTrivCond}
For a given finitely-supported multiplicity function $\aleph$, $T_{\aleph}\neq\{0\}$ if and only if one of the following two conditions holds.
\begin{enumerate}
\item $\aleph(p,n)=0$ for all $p\in\supp{\aleph}$ and $n>1$, and $\mathcal{X}_{\aleph}\neq\varnothing$. In this case,
\[T_{\aleph}=z_0\mathbb{Z},\]
where $z_0=2\pi/\omega_0$ and $\omega_0$ is an element of $\mathcal{X}_{\aleph}$ such that $|\omega_0|=\max{\left\{|\omega|\colon \omega\in\mathcal{X}_{\aleph}\right\}}$.
\item $\supp{\aleph}=\{p_0\}$ and $x_{p_0}=0$. In this case $T_{\aleph}=\mathbb{C}$.
\end{enumerate}
\begin{proof}
The second case follows immediately from the definitions of $T_{\aleph}$ and the matrix exponential. We prove that the first case is the only remaining case.

Note that it is clear that $0\in T_{\aleph}$. We can decompose the exponential $e^{zJ(\aleph)}$ as a direct sum:
\[e^{zJ(\aleph)}=\bigoplus_{p\in\supp{\aleph}}{\bigoplus_{n=1}^{\infty}{\bigoplus_{\aleph(p,n)}{e^{zJ(p,n)}}}}.\]
Recall that $J(p,n)$ is the $n\times n$ matrix $x_p \mathbbm{1} + N_n$, where $x_p$ is a root of the polynomial $p$ and $N_n$ is nilpotent with ones above the main diagonal and zeros elsewhere. Since the commutator of $x_p \mathbbm{1}$ and $N_n$ vanishes, we have
\begin{align} \label{ExpJordExpansion}
e^{zJ(p,n)}&=e^{z(x_p\mathbbm{1} + N_n)}\notag\\
&=e^{zx_p}e^{zN_n}\notag\\
&=e^{zx_p}\left(\mathbbm{1} + zN_n + \frac{z^2}{2!} N_n^2 + \dots + \frac{z^{n-1}}{(n-1)!}N_{n}^{n-1}\right).
\end{align}
Note that $e^{zJ(\aleph)}$ is the identity if and only if the exponential of each Jordan block $e^{zJ(p,n)}$ are themselves the identity. By the expansion \eqref{ExpJordExpansion}, we can see that the exponentials of the Jordan blocks are the identity precisely when the higher order terms vanish and $e^{zx_p} = 1$. First, we study when the higher order terms vanish.

When $n = 1$, we have that $N_1$ is the $1 \times 1$ matrix $[0]$, so the higher order terms vanish irrespective of our choice of $z$. Suppose that $n>1$ and there exists $z$ such that the higher order terms vanish:
\begin{equation} \label{HighOrdTermVanish}
zN_n+\frac{z^2}{2!}N_n^2+\dots+\frac{z^{n-1}}{(n-1)!}N_n^{n-1}=[0]_n.\end{equation}
Since $n>1$, the second column of $N_n$ consists of $1$ in the first component and zeros elsewhere. It follows that for the entry in the first row, second column of both sides of \eqref{HighOrdTermVanish} to match, we must have $z\cdot1=z=0$. Hence, nontrivial solutions to $e^{zJ(\aleph)} = \mathbbm{1}$ can exist only if higher order terms vanish independently of $z$, which can only occur if $\aleph$ vanishes for $n>1$ so that the only nilpotent matrix we deal with is $N_1$.

Restricting ourselves to $\aleph$ that vanishes for $n>1$ and equating \eqref{ExpJordExpansion} with $\mathbbm{1}$ gives us
\begin{equation*}
    e^{zx_p} \mathbbm{1} = \mathbbm{1}.
\end{equation*}
So we must have $e^{zx_p} = 1$. In particular, we must have this equation hold \textit{for all} $p\in\supp{\aleph}$ so that all of the Jordan blocks are the identity. Symbolically, we have established that $T_{\alpha}\neq\{0\}$ if and only if
\[\forall n>1,\aleph(p,n)=0\text{ and }\exists z\neq0\text{ s.t. }\forall p\in\supp{\aleph},\ e^{zx_p}=1.\]
Now we show that
\[\exists z\neq0\text{ s.t. }\forall p\in\supp{\aleph},e^{zx_p}=1\Longleftrightarrow\mathcal{X}_{\aleph}\neq\emptyset.\]
In one direction, suppose that for all $p\in\supp{\aleph}$, $e^{zx_p}=1$, and $z\neq0$ is arbitrary. This implies that for any $p\in\supp{\aleph}$, we can find an integer $N_p$ such that $zx_p=2\pi iN_p$. Hence $\frac{2\pi}{z}\in\mathcal{X}_{\aleph}$ and so $\mathcal{X}_{\aleph}$ is nonempty.

In the other direction, suppose $\mathcal{X}_{\aleph}$ is nonempty, with $\omega\in\mathcal{X}_{\aleph}$. Observe that $\omega=0$ would imply that $\supp{\aleph}=\{0\}$. Since we restrict ourselves to $\aleph$ that vanishes for $n>1$, we would have that $J(\aleph)$ is the zero matrix, which is impossible since our Lie algebra is almost Abelian. Thus $\omega\neq0$ and we can set $z=\frac{2\pi}{\omega}$. By definition, for every $p\in\supp{\aleph}$ there exists integers $N_p$ such that
\[x_p=\frac{2\pi i  N_p}{z}\Longrightarrow zx_p=2\pi i  N_p\Longrightarrow e^{zx_p}=1,\]
which completes the last direction.

Observe that the map $f\colon z\mapsto e^{zJ(\aleph)}$ is a Lie group homomorphism and $T_{\aleph}$ is precisely the kernel of this homomorphism. Since $\{\mathbbm{1}\}$ is discrete and $f$ is continuous, we must have that $T_{\aleph}=f^{-1}(\{\mathbbm{1}\})$ is a discrete subgroup of $\mathbb{C}$. So $T_{\aleph}$ is a lattice of the form
\[T_{\aleph}=z_0\mathbb{Z}\oplus w_0\mathbb{Z},\]
where $z_0$ and $w_0$ are $\mathbb{R}$-linearly independent as long as both are nonzero. Since $z_0=w_0=0$ yields the degenerate case $T_{\aleph}=\{0\}$, if $T_{\aleph}\neq\{0\}$, at least one of $z_0$ and $w_0$ is nonzero, so $T_{\aleph}\neq\{0\}$ is isomorphic to either $\mathbb{Z}$ or $\mathbb{Z}^2$.

Suppose $p\in\supp{\aleph}$ and $z\in T_{\aleph}$ is nonzero (so $T_{\aleph}$ is nontrivial). We have $e^{zx_p}=1$. So there exists a nonzero integer $N$ such that $zx_p=2\pi iN$. Now pick $w\in T_{\aleph}$ nonzero. Since $e^{wx_p}=1$, there exists an integer $M$ such that $w=\frac{2\pi i  M}{x_p}$. It follows that $w=\frac{M}{N}z$. Therefore, if $T_{\aleph}$ is nontrivial, all of its elements are colinear in the complex plane. In particular, $T_{\aleph}\ncong\mathbb{Z}^2$. Nontrivial $T_{\aleph}$ thus take the form
\[T_{\aleph}=z_0\mathbb{Z},\]
where $z_0$ is a complex number in $T_{\aleph}$ that has the smallest positive magnitude. Since $|z_0|$ is minimal amongst the nonzero elements of $T_{\aleph}$, $\frac{2\pi}{|z_0|}$ is maximal amongst elements of $\mathcal{X}_{\aleph}$. Hence,
\[z_0=\frac{2\pi}{\omega_0},\quad \omega_0\in\mathcal{X}_{\aleph}\text{ such that }|\omega_0|=\max{\left\{|\omega|\colon \omega\in\mathcal{X}_{\aleph}\right\}},\]
and we are done.
\end{proof}
\end{lemma}

The upshot of Lemma \ref{TalephTrivCond} is that in the interesting cases, $\rank{T_{\aleph}}\leq 1$. This condition comes into play in the proof of Prop \ref{CompCondition}, which characterizes when a connected Lie subgroup of an almost Abelian Lie group is compact.


\section{Group Representations and Corresponding Exponential Maps}
The core results of this paper depend on some convenient matrix representations of almost Abelian Lie groups. Given an almost Abelian Lie algebra $\aA(\aleph)$ of dimension $d+1$, we recall from Prop. 2 in \cite{Ave16} that we have the matrix representation
\begin{equation} \label{ConnAAAlgRep}
    \aA(\aleph) = \left\{ \begin{pmatrix}
        0 & 0 \\
        v & tJ(\aleph)
    \end{pmatrix}       
     : (v,t) \in \C^d \oplus \C \right\}.
\end{equation}
Looking at the exponential of this matrix representation, we can conjecture a matrix representation (that of Prop. \ref{FinConnAAGroupProp}) for a \textit{connected} almost Abelian Lie group. However, as Proposition \ref{FinConnAAGroupProp} will show, this representation unfortunately is often not \textit{simply connected}. So, we will use this representation and the calculation of the corresponding matrix exponential (Lemma \ref{lem1}) as intuition for the simply connected representation (Prop. \ref{SimpConnRep}) and corresponding matrix exponential (Lemma \ref{SimpConnExp}) which we need.

\begin{proposition} \label{FinConnAAGroupProp}
For a finitely-supported multiplicity function $\aleph$, let
\begin{equation*}
    G \coloneqq \left\{ \begin{pmatrix}
        1 & 0 \\
        v & e^{t J(\aleph)}
    \end{pmatrix} \middle| \; (v,t) \in \C^d \oplus \C \right\}.
\end{equation*}
Then $G$ is a connected Lie group with Lie algebra $\aA(\aleph)$, and it is simply connected if and only if $T_\aleph = \{0\}$.
\begin{proof}
That $G$ is a connected Lie group because every element is path connected to the identity is clear from the definition. Then for all $(u,s) \in \C^d \oplus \C$, let $\gamma_{(u,s)}: (-1,1) \to G$ be a smooth curve defined by
\begin{equation*}
    \gamma_{(u,s)}(\tau) \coloneqq \begin{pmatrix}
        1 & 0 \\
        v(\tau) & e^{t(\tau) J(\aleph)}
    \end{pmatrix},
\end{equation*}
with 
\begin{equation*}
    (v(0), t(0)) = (0,0), \qquad (v'(0), t'(0)) = (u,s).
\end{equation*}
Then since we can split the derivative of $e^{t(\tau) J(\aleph)}$ into its real and complex parts, we may calculate
\begin{equation*}
    \dv{}{\tau} \begin{pmatrix}
        1 & 0 \\
        v(\tau) & e^{t(\tau) J(\aleph)}
    \end{pmatrix} \biggr|_{\tau=0} = 
    \begin{pmatrix}
        0 & 0 \\
        u & sJ(\aleph)
    \end{pmatrix}  \in \aA(\aleph),
\end{equation*}
where the last inclusion is follows from Prop. 3 of \cite{Ave16}. Thus $\aA(\aleph)$ is the Lie algebra of $G$.

Consider the map $\varphi: \C^d \times \C \to G$, defined by:
\[ (v,t)\mapsto\begin{pmatrix}
        1 & 0 \\
        v & e^{t J(\aleph)}
    \end{pmatrix}.\]

Let $\pi\colon\mathbb{C}^d\times\mathbb{C}\to\mathbb{C}^d\times\left(\C/T_\aleph\right)$ be the natural quotient map. In particular, we define an equivalence relation $\sim$ on $\mathbb{C}$ where $t\sim t'$ if and only if $t-t'\in T_{\aleph}$. Then $\pi$ maps $(v,t)\mapsto(v,[t])$ where $[t]$ is the equivalence class of $t$ under this relation.

Suppose that $t \sim t'$. Then, $e^{tJ(\aleph)}=e^{t'J(\aleph)}$ so that $\varphi(v,t)=\varphi(v,t')$. Now, we may define the map $\psi \colon \mathbb{C}^d \times \left( \C/T_\aleph \right)\to G$ that maps
\[\left(v,[t]\right)\mapsto\begin{pmatrix}
1 & 0\\
v & e^{tJ(\aleph)}
\end{pmatrix}.\]
$\psi$ is smooth with a smooth inverse. So $G$ is diffeomorphic to $\mathbb{C}^d\times\left(\C/T_\aleph\right)$.

Suppose $T_{\aleph}$ is trivial. Then $G$ is diffeomorphic to $\mathbb{C}^{d+1}$, which is simply connected.

On the other hand, suppose that $G$ is simply connected. By Lemma \ref{TalephTrivCond}, we have that either $T_{\aleph}$ is trivial or $T_{\aleph}\cong\mathbb{Z}$.

If $T_{\aleph}\cong\mathbb{Z}$, we have 
\[
\C/T_\aleph\cong \C/\Z \cong \R \times (\R/\Z)
\]
But the map $t\mapsto e^{2\pi i  t}$ is a homomorphism from $\mathbb{R}$ to $S^1$ and $\mathbb{Z}$ is the kernel of the homomorphism, hence $\R/\Z\cong S^1$ and so $G$ is diffeomorphic to $\mathbb{R}^{2d+1}\times S^1$. In particular, the fundamental group of $G$ is $\pi_1(G)\cong\pi_1(S^1)\cong\mathbb{Z}$, so $G$ is not simply connected, a contradiction. Therefore, $T_{\aleph}\ncong\mathbb{Z}$.

We conclude that $G$ is simply connected if and only if $T_{\aleph}$ is trivial.
\end{proof}
\end{proposition}

We now calculate matrix exponential on the almost Abelian Lie algebra representation of \eqref{ConnAAAlgRep} and see that it lands in the group representation of Prop. \ref{FinConnAAGroupProp}.


\begin{lemma} \label{lem1}
\textit{The matrix exponential map of the matrix Lie algebra $\aA(\aleph)$ represented as in Prop. \ref{FinConnAAGroupProp} is given by} 
\begin{equation*}
    \exp \begin{pmatrix}
    0 & 0 \\
    v & t J(\aleph)
    \end{pmatrix} =
    \begin{pmatrix}
        1 & 0 \\
        \frac{e^{tJ(\aleph)} - \mathbbm{1}}{t J(\aleph)} v & e^{t J(\aleph)}
    \end{pmatrix}.
\end{equation*}
    
\begin{proof}
First, we show that 
\begin{equation} \label{wtsLem1}
        \begin{pmatrix}
            0 & 0 \\
            v & tJ(\aleph)
        \end{pmatrix}^n 
        =
        \begin{pmatrix}
            0 & 0 \\
            [tJ(\aleph)]^{n-1} v & [tJ(\aleph)]^n
        \end{pmatrix},
\end{equation}
for all $n\in\mathbb{N}$ by inducting on $n$. For $n = 1$, we indeed have
\begin{equation} \label{basecaseLem1}
        \begin{pmatrix}
            0 & 0 \\
            v & tJ(\aleph)
        \end{pmatrix}^1 = 
        \begin{pmatrix}
            0 & 0 \\
            [tJ(\aleph)]^{0} v & [tJ(\aleph)]^1
        \end{pmatrix},
    \end{equation}
and thus \eqref{basecaseLem1} is our inductive base case. Assume \eqref{wtsLem1} is true for $n = k \in \N$. We show that \eqref{wtsLem1} holds for $k+1$:
\begin{equation*}
        \begin{pmatrix}
            0 & 0 \\
            v & tJ(\aleph)
        \end{pmatrix}^{k+1}
        = 
        \begin{pmatrix}
            0 & 0 \\
            [tJ(\aleph)]^{k-1} v & [tJ(\aleph)]^{k}
        \end{pmatrix}
        \begin{pmatrix}
            0 & 0 \\
            v & tJ(\aleph)
        \end{pmatrix}
        =
        \begin{pmatrix}
            0 & 0 \\
            [tJ(\aleph)]^{k} v & [tJ(\aleph)]^{k+1}
        \end{pmatrix}.
    \end{equation*}
Thus by induction \eqref{wtsLem1} holds for all $n \in \N$. 

Now, by the series expansion of the matrix exponential, we have:
\begin{align*}
        \exp\begin{pmatrix}
            0 & 0 \\
            v & tJ(\aleph)
        \end{pmatrix} &= \sum_{n=0}^\infty \frac{1}{n!} \begin{pmatrix}
            0 & 0 \\
            v & tJ(\aleph)
        \end{pmatrix}^n \\
        &= \mathbbm{1} + \sum_{n=1}^\infty \frac{1}{n!}  
        \begin{pmatrix}
            0 & 0 \\
            [tJ(\aleph)]^{n-1} v & [tJ(\aleph)]^n
        \end{pmatrix} \\
        &= \begin{pmatrix}
            1 & 0 \\
            \frac{e^{tJ(\aleph)} - \mathbbm{1}}{t J(\aleph)} v & e^{t J(\aleph)}
        \end{pmatrix},
    \end{align*}
where the last equality comes from the component-wise series expansions, and the term $\frac{e^{tJ(\aleph)} - \mathbbm{1}}{t J(\aleph)}$ denotes the series of the matrix exponential, subtracted by the identity matrix, and with one less power of the argument, $tJ(\aleph)$, in each summed term.
\end{proof}
\end{lemma}


We now find a representation for the simply connected almost Abelian Lie group corresponding to a given almost Abelian Lie algebra $\aA(\aleph)$.

\begin{proposition} \label{SimpConnRep}
For a finitely-supported multiplicity function $\aleph$, let
\begin{equation*}
    G \coloneqq \left\{ 
    \begin{pmatrix}
    1 & 0 & 0 \\
    v & e^{tJ(\aleph)} & 0 \\
    t & 0 & 1
    \end{pmatrix} 
    \: \middle| \ (v,t) \in \C^d \oplus \C \right\}.
\end{equation*}
Then $G$ is a complex simply connected Lie group with Lie algebra isomorphic to $\aA(\aleph)$.

\begin{proof} Note that Prop. 3 in \cite{Ave16} showed that a finite-dimensional almost Abelian Lie algebra $\aA(\aleph)$ corresponding to a finite dimensional multiplicity function $\aleph: \C \times \N \to \N$ has a faithful matrix representation:
\begin{equation} \label{OGLieAlgRep}
    \aA(\aleph) \cong \C^d \rtimes \C \ni (v,t) \mapsto \begin{pmatrix}
        0 & 0 \\
        v & tJ(\aleph)
    \end{pmatrix}.
\end{equation}
Moreover, note that the map $\Phi$ defined by
\begin{equation} \label{NewLieAlgRep}
    \aA(\aleph) \ni \begin{pmatrix}
        0 & 0 \\
        v & tJ(\aleph)
    \end{pmatrix} 
    \mapsto
    \begin{pmatrix}
    0 & 0 & 0 \\
    v & tJ(\aleph) & 0 \\
    t & 0 & 0
    \end{pmatrix}
\end{equation}
is a complex Lie algebra isomorphism, and so we have another faithful matrix representation of $\aA(\aleph)$. For completeness, we check that this is indeed a Lie algebra isomorphism. It is clear that the map is bijective, so it suffices to check that it preserves Lie brackets. We compute
\begin{align*}
    \Phi\left[ \begin{pmatrix}
        0 & 0 \\
        v & tJ(\aleph)
    \end{pmatrix}, 
    \begin{pmatrix}
        0 & 0 \\
        u & sJ(\aleph)
    \end{pmatrix} \right]
    &= \Phi \left( \begin{pmatrix}
        0 & 0 \\
        tJ(\aleph)u & ts(J(\aleph))^2
    \end{pmatrix} 
    - \begin{pmatrix}
        0 & 0 \\
        sJ(\aleph)v & ts(J(\aleph))^2
    \end{pmatrix} \right)
    \\
    &= \Phi \begin{pmatrix}
        0 & 0 \\
        tJ(\aleph)u - sJ(\aleph)v & 0
    \end{pmatrix}
    \\
    &= \begin{pmatrix}
        0 & 0 & 0 \\
        tJ(\aleph)u - sJ(\aleph)v & 0 & 0\\
        0 & 0 & 0
    \end{pmatrix} 
    \\
    &= \begin{pmatrix}
        0 & 0 & 0 \\
        tJ(\aleph)u & ts(J(\aleph))^2 & 0 \\
        0 & 0 & 0
    \end{pmatrix}
    -
    \begin{pmatrix}
        0 & 0 & 0 \\
        sJ(\aleph)v & ts(J(\aleph))^2 & 0 \\
        0 & 0 & 0
    \end{pmatrix}
    \\
    &= \left[\Phi\begin{pmatrix}
        0 & 0 \\
        v & tJ(\aleph)
    \end{pmatrix},
    \Phi\begin{pmatrix}
        0 & 0 \\
        u & sJ(\aleph)
    \end{pmatrix} \right].
\end{align*}
We define $\Phi\inv$ by
\begin{equation*}
    \Phi\inv\begin{pmatrix}
        0 & 0 & 0 \\
        v & tJ(\aleph) & 0 \\
        t & 0 & 0
    \end{pmatrix} 
    =
    \begin{pmatrix}
        0 & 0 \\
        v & tJ(\aleph)
    \end{pmatrix}.
\end{equation*}
Then we check:
\begin{align*}
    &\quad\Phi\inv \left[ 
    \begin{pmatrix}
        0 & 0 & 0 \\
        v & tJ(\aleph) & 0 \\
        t & 0 & 0
    \end{pmatrix},
    \begin{pmatrix}
        0 & 0 & 0 \\
        u & sJ(\aleph) & 0 \\
        s & 0 & 0
    \end{pmatrix}
    \right]\\
    &=
    \Phi\inv \left( \begin{pmatrix}
        0 & 0 & 0 \\
        v & tJ(\aleph) & 0 \\
        t & 0 & 0
    \end{pmatrix}
    \begin{pmatrix}
        0 & 0 & 0 \\
        u & sJ(\aleph) & 0 \\
        s & 0 & 0
    \end{pmatrix}
    -
    \begin{pmatrix}
        0 & 0 & 0 \\
        u & sJ(\aleph) & 0 \\
        s & 0 & 0
    \end{pmatrix}
    \begin{pmatrix}
        0 & 0 & 0 \\
        v & tJ(\aleph) & 0 \\
        t & 0 & 0
    \end{pmatrix}
    \right) \\
    &=
    \Phi\inv \left(\begin{pmatrix}
        0 & 0 & 0 \\
        tJ(\aleph)u & ts(J(\aleph))^2 & 0 \\
        0 & 0 & 0
    \end{pmatrix}
    -
    \begin{pmatrix}
        0 & 0 & 0 \\
        sJ(\aleph)v & ts(J(\aleph))^2 & 0 \\
        0 & 0 & 0
    \end{pmatrix} 
    \right) \\
    &= 
    \Phi\inv \left( \begin{pmatrix}
        0 & 0 & 0 \\
        tJ(\aleph)u - sJ(\aleph)v & 0 & 0 \\
        0 & 0 & 0
    \end{pmatrix}
    \right) \\
    &= 
    \begin{pmatrix}
        0 & 0 \\
        tJ(\aleph)u - sJ(\aleph)v & 0
    \end{pmatrix} \\
    &= \begin{pmatrix}
        0 & 0 \\
        tJ(\aleph)u & ts(J(\aleph))^2
    \end{pmatrix}
    - 
    \begin{pmatrix}
        0 & 0 \\
        sJ(\aleph)v & ts(J(\aleph))^2
    \end{pmatrix} \\
    &= [\Phi\inv(v,t), \Phi\inv(u,s)].
\end{align*}

That $G$ is a closed subset of $\mathrm{GL}_n(\C)$ is apparent from its definition. That it is closed under multiplication is verified in the course of the proof of Prop. \ref{CenterSimpConn} below. That every element has an inverse is seen by observing that:
\begin{equation*}
    \begin{pmatrix}
    1 & 0 & 0 \\
    v & e^{t J(\aleph)} & 0 \\
    t & 0 & 1
    \end{pmatrix}
    \begin{pmatrix}
    1 & 0 & 0 \\
    e^{-tJ(\aleph)}(-v) & e^{-t J(\aleph)} & 0 \\
    -t & 0 & 1
    \end{pmatrix}
    =
    \mathbbm{1}.
\end{equation*}
Thus $G$ is a complex Lie group as a closed subgroup of $\mathrm{GL}_n(\C)$. Consider the map $\varphi: \C^d \oplus \C \to G$ given by 
\begin{equation*}
    \varphi(v,t) \coloneqq 
    \begin{pmatrix}
    1 & 0 & 0 \\
    v & e^{t J(\aleph)} & 0 \\
    t & 0 & 1
    \end{pmatrix}.
\end{equation*}
This is certainly injective because $(v,t)$ is represented in the image. Furthermore, it is easily seen to be surjective by the definition of $G$. Since $\varphi$ is obviously a holomorphism, we have that $\varphi$ is a biholomorphism. Thus $G \cong_{\text{biholo.}} \C^d \oplus \C$, so $G$ is simply connected.

Now consider a path $\gamma_{(u,s)} : (-1,1) \to G$ defined by 
\begin{equation*}
    \gamma(\tau) \coloneqq \begin{pmatrix}
    1 & 0 & 0 \\
    v(\tau) & e^{t(\tau) J(\aleph)} & 0 \\
    t(\tau) & 0 & 1
    \end{pmatrix},
\end{equation*}
with $(v(0), t(0)) = (0,0)$ and $(v'(0), t'(0)) = (u,s)$. Then
\begin{equation*}
    \dv{}{\tau} \begin{pmatrix}
    1 & 0 & 0 \\
    v(\tau) & e^{t(\tau) J(\aleph)} & 0 \\
    t(\tau) & 0 & 1
    \end{pmatrix} \biggr|_{\tau=0}
    = 
    \begin{pmatrix}
    0 & 0 & 0 \\
    v & J(\aleph) & 0 \\
    t & 0 & 0
    \end{pmatrix} \in \aA(\aleph),
\end{equation*}
where the inclusion at the end follows from \eqref{NewLieAlgRep}. Thus $\text{Lie}(G) \cong \aA(\aleph)$ by the faithful representation of \eqref{NewLieAlgRep}.
\end{proof}
\end{proposition}

Having found a faithful matrix representation for simply connected almost Abelian Lie groups, it is convenient for us to find the exponential map corresponding to this faithful representation. In the complex case, different matrix representations of the Lie algebra may yield different identities between the geometric exponential maps and the matrix exponential.

\begin{proposition} \label{SimpConnExp}
For a complex simply connected almost Abelian group $G$ with Lie algebra $\aA(\aleph)$, the exponential map $\exp:\aA(\aleph) \to G$ is given by
$$\exp \begin{pmatrix}
0 & 0 & 0 \\
v & tJ(\aleph) & 0 \\
t & 0 & 0 \\
\end{pmatrix} = \begin{pmatrix}
1 & 0 &0 \\
\frac{e^{tJ(\aleph)} - \mathbbm{1}}{tJ(\aleph)}v & e^{tJ(\aleph)} & 0 \\
t & 0 & 1 
\end{pmatrix} \in G.$$
\begin{proof}
We first show that 
\begin{equation} \label{indhypSimpConnExp}
    \begin{pmatrix}
0 & 0 & 0 \\
v & tJ(\aleph) & 0 \\
t & 0 & 0 \\
\end{pmatrix}^n = \begin{pmatrix}
0 & 0 & 0 \\
[tJ(\aleph)]^{n-1}v & [tJ(\aleph)]^n & 0 \\
0 & 0 & 0 \\
\end{pmatrix},
\end{equation}
for all integers $n\geq2$. We proceed by inducting on $n$. When $n=2$, we have 
\begin{equation} \label{basecaseSimpConnExp}
    \begin{pmatrix}
0 & 0 & 0 \\
v & tJ(\aleph) & 0 \\
t & 0 & 0 \\
\end{pmatrix}^2 = \begin{pmatrix}
0 & 0 & 0 \\
v & tJ(\aleph) & 0 \\
t & 0 & 0 \\
\end{pmatrix}\begin{pmatrix}
0 & 0 & 0 \\
v & tJ(\aleph) & 0 \\
t & 0 & 0 \\
\end{pmatrix} = \begin{pmatrix}
0 & 0 & 0 \\
[tJ(\aleph)]^1v & [tJ(\aleph)]^2 & 0 \\
0 & 0 & 0 \\
\end{pmatrix},
\end{equation}
so the base case holds. Next, assume \eqref{basecaseSimpConnExp} is true when $n=k$ for some $k \in \N$. When $n=k+1$,
\begin{align*}
    \begin{pmatrix}
0 & 0 & 0 \\
v & tJ(\aleph) & 0 \\
t & 0 & 0 \\
\end{pmatrix}^{k+1} &= \begin{pmatrix}
0 & 0 & 0 \\
v & tJ(\aleph) & 0 \\
t & 0 & 0 \\
\end{pmatrix}^k \begin{pmatrix}
0 & 0 & 0 \\
v & tJ(\aleph) & 0 \\
t & 0 & 0 \\
\end{pmatrix} \\
&= \begin{pmatrix}
0 & 0 & 0 \\
[tJ(\aleph)]^{k-1}v & [tJ(\aleph)]^k & 0 \\
0 & 0 & 0 \\
\end{pmatrix} \begin{pmatrix}
0 & 0 & 0 \\
v & tJ(\aleph) & 0 \\
t & 0 & 0 \\
\end{pmatrix} \\
&= \begin{pmatrix}
0 & 0 & 0 \\
[tJ(\aleph)]^kv & [tJ(\aleph)]^{k+1} & 0 \\
0 & 0 & 0 \\
\end{pmatrix}.
\end{align*}
This completes the induction.

By the series expansion of the matrix exponential we have,
\begin{align*}
    \exp \begin{pmatrix}
0 & 0 & 0 \\
v & tJ(\aleph) & 0 \\
t & 0 & 0 \\
\end{pmatrix} 
&= \sum_{n=0}^\infty \frac{1}{n!} \begin{pmatrix}
0 & 0 & 0 \\
v & tJ(\aleph) & 0 \\
t & 0 & 0 \\
\end{pmatrix}^n \\
&= \mathbbm{1} + \begin{pmatrix}
0 & 0 & 0 \\
v & tJ(\aleph) & 0 \\
t & 0 & 0 \\
\end{pmatrix} + \sum_{n=2}^\infty \frac{1}{n!} \begin{pmatrix}
0 & 0 & 0 \\
[tJ(\aleph)]^{n-1}v & [tJ(\aleph)]^n & 0 \\
0 & 0 & 0 \\
\end{pmatrix} \\
&= \begin{pmatrix}
1 & 0 & 0 \\
\frac{e^{tJ(\aleph)} - \mathbbm{1}}{tJ(\aleph)}v & e^{tJ(\aleph)} & 0 \\
t & 0 & 1 \\
\end{pmatrix}.
\end{align*}
The last equality comes from the component-wise series expansions, and the term $\frac{e^{tJ(\aleph)} - \mathbbm{1}}{t J(\aleph)}$ denotes the series of the matrix exponential, subtracted by the identity matrix, and with one less power of the argument, $tJ(\aleph)$, in each summed term. Note that $\frac{e^{tJ(\aleph)} - \mathbbm{1}}{tJ(\aleph)}v \in \C^d$, so it follows that 
$$\begin{pmatrix}
1 & 0 & 0 \\
\frac{e^{tJ(\aleph)} - \mathbbm{1}}{tJ(\aleph)}v & e^{tJ(\aleph)} & 0 \\
t & 0 & 1 \\
\end{pmatrix} \in G.$$
\end{proof}
\end{proposition}

Finally, we note that as a consequence of Lemma \ref{SimpConnExp}, the exponential map is particularly simple to understand on the Abelian subalgebra of an almost Abelian Lie algebra.

\begin{remark} \label{ExpIsId}
Let $G$ be the simply connected group that has Lie algebra $\aA(\aleph)$. It follows that on the Abelian Lie subalgebra $\ker(J(\aleph)) \oplus \C$ the exponential map $\exp: \aA(\aleph) \to G$ associated with $G$ is given by:
\begin{equation*}
    \exp(v,t) = [v,t], \qquad \forall (v,t) \in \ker(J(\aleph)) \oplus \C.
\end{equation*}
\begin{proof} If $v \in \ker(J(\aleph))$ then:
\begin{align*}
    \frac{e^{tJ(\aleph)} - \mathbbm{1}}{t J(\aleph)}v &= \left( \sum_{n=1}^\infty \frac{1}{n!} (t J(\aleph))^{n-1} \right) v \\
    &= v + \sum_{n=2}^\infty \frac{1}{n!} t^{n-1} \big( (J(\aleph))^{n-1} v \big) \\
    &= v.
\end{align*}
\end{proof}
\end{remark}


\section{The Center of a Complex Almost Abelian Group}

We recall the following standard fact from Lie group theory.

\begin{lemma} \label{ExpMapToCenter}
Let $\mathfrak{g}$ be an arbitrary Lie algebra, $G$ be a connected matrix Lie group that has Lie algebra $\mathfrak{g}$, and let $\exp_G: \mathfrak{g} \to G$ be the corresponding exponential map (specific to $G$). Then $\exp_G(Z(\mathfrak{g})) \subseteq Z(G)$.
\end{lemma}

\begin{proposition} \label{CenterSimpConn}
Let $G$ be a simply connected almost Abelian Lie group with Lie algebra $\aA(\aleph)$. Recall Definition \ref{TAlephDef}. The center of $G$ is given by:
\begin{align*}
    Z(G) &= \exp_{G}[Z(\aA(\aleph))] \times T_\aleph \\
    &= \exp_{G}[Z(\aA(\aleph)) \times T_\aleph] \\
    &= \{(u,s) \in \C^d \rtimes \C \divides u \in \ker(J(\aleph)),\ e^{s J(\aleph)} = \mathbbm{1} \}
\end{align*}
where $\exp_{G}: \aA(\aleph) \to G$ is the associated exponential map with $G$.

Also, the preimage under the exponential map (associated with $G$) of the identity component of the center is:
\begin{equation*}
    \exp_{G}\inv[Z(G)_0] = Z(\aA(\aleph))
\end{equation*}

\begin{proof}
We represent the standard matrix exponential that is a matrix series as $e^A$ where $A$ is understood to be a matrix. By Prop. \ref{SimpConnRep} we may use the representation:
\begin{equation*}
    G \coloneqq \left\{ \begin{pmatrix}
        1 & 0 & 0 \\
        v & e^{t J(\aleph)} & 0 \\
        t & 0 & 1
    \end{pmatrix} \middle| \ (v,t) \in \C^d \oplus \C \right\}.
\end{equation*}
For simplicity, we represent an element of $G$ with this matrix representation by a bracket-tuple as follows:
\begin{equation*}
    [v,t] \coloneqq \begin{pmatrix}
        1 & 0 & 0 \\
        v & e^{t J(\aleph)} & 0 \\
        t & 0 & 1
    \end{pmatrix}. 
\end{equation*}

Thus we may compactly represent the multiplication of group elements by:
\begin{equation}\label{multiplicationOfAAGpElsEq}
    [v,t][u,s] = [v+e^{tJ(\aleph)}u,t+s].
\end{equation}

Now suppose $[u,s] \in Z(G)$, so $[v,t][u,s] = [u,s][v,t]$. Then by \eqref{multiplicationOfAAGpElsEq}:
\begin{equation*}
    [u, s][v, t] = [u + e^{sJ(\aleph)} v,s+t].
\end{equation*}
Thus the condition $[v,t][u,s] = [u,s][v,t]$ is equivalent to $v + e^{tJ(\aleph)}u = u + e^{sJ(\aleph)} v$. We can then rewrite this latter expression as
\begin{equation}\label{commutingCondition}
    (e^{sJ(\aleph)} - \mathbbm{1})v = (e^{tJ(\aleph)} - \mathbbm{1})u.
\end{equation}
Setting $v = 0$ we have that $(e^{tJ(\aleph)} - \mathbbm{1})u = 0$ and thus we must have $J(\aleph)u = 0$, that is, we have $u \in \ker(J(\aleph))$, as desired. 

Since we are supposing $[u,s] \in Z(G)$, equation \eqref{commutingCondition} must hold for all $v$. So if $u \in \ker(J(\aleph))$, then $(e^{sJ(\aleph)} - \mathbbm{1})v = 0$ for all $v$, which means that $e^{sJ(\aleph)} = \mathbbm{1}$. Thus we have proven 
\begin{equation*}
    Z(G) \subseteq  \{[u,s] \in \C^d \rtimes \C \divides u \in \ker(J(\aleph)),\ e^{s J(\aleph)} = \mathbbm{1} \}.
\end{equation*}

For notational convenience, define 
\begin{equation*}
    X \coloneqq \{(u,s) \in \C^d \oplus \C = \aA(\aleph) \divides u \in \ker(J(\aleph))  = Z(\aA(\aleph)),\ e^{s J(\aleph)} = \mathbbm{1} \},
\end{equation*}
where we recall (Remark 2 in \cite{Ave16}) that $\ker(J(\aleph))  = Z(\aA(\aleph))$.

Now suppose $(u,s) \in  X$. By the conditions of the set $X$ on the last component of an element belonging to it, $s \in T_\aleph$ by definition. Thus $X = Z(\aA(\aleph)) \times T_\aleph$. By Remark \ref{ExpIsId}, 
$\exp_{G} ((v,t)) = [v,t] \foreach (v,t) \in Z(\aA(\aleph)) \times T_{\aleph}$. Notice that $[v,t] = [v,0][0,t]$, thus
\begin{equation*}
    \exp_{G}\big( Z(\aA(\aleph)) \times T_\aleph \big) = \exp_{G}\big(Z(\aA(\aleph))\big) \times T_\aleph.
\end{equation*}

Now by Lemma \ref{ExpMapToCenter}, we have $\exp_{G}\big( Z(\aA(\aleph)) \big) \subseteq Z(G)$. Then we calculate: let $[v,t] \in G,\ [u,s] \in \exp_{G}\big( Z(\aA(\aleph)) \big) \times T_\aleph$. Then 
\begin{align*}
    [v,t][u,s] &= [v + e^{tJ}u, t+s] \\
    &= [v + \left( \sum_{n=0}^\infty \frac{1}{n!} (tJ)^n \right) u , t + s] \\
    &= [v + \left( \sum_{n=0}^\infty \frac{1}{n!} (tJ)^n u \right), t + s] \\
    &= [v + u, t + s] = [u + v, s + t] = [u,s] [v,t].
\end{align*}
Thus $\exp_{G}\big( Z(\aA(\aleph)) \big) \times T_\aleph \subseteq Z(G)$. 

By Remark \ref{ExpIsId}, we have that as sets 
\begin{equation*}
    \exp_{G}\big( Z(\aA(\aleph)) \big) = Z(\aA(\aleph)) = \ker(J(\aleph)).
\end{equation*}

Thus by bidirectional inclusion, 
\begin{equation*}
    Z(G) = \{[u,s] \in \C^d \rtimes \C \divides u \in \ker(J(\aleph)),\ e^{s J(\aleph)} = \mathbbm{1} \}.
\end{equation*}

Thus we have:
\begin{align*}
    Z(G) &= \{[u,s] \in \C^d \rtimes \C \divides u \in \ker(J(\aleph)),\ e^{s J(\aleph)} = \mathbbm{1} \} \\
    &= \exp_{G}\big( Z(\aA(\aleph)) \big) \times T_\aleph. \\
    &= \exp_{G}\big( Z(\aA(\aleph)) \times T_\aleph \big)
\end{align*}

Finally, if $\exp(v,t)=[u,s]\in Z(G)_0$, then since we have $Z(G)=\exp_{G}(\aA(\aleph))\times T_\aleph$, we have by connectedness that $Z(G)_0 = \exp_{G}(\aA(\aleph)) \times \{0\}$. Hence $v=u$ and $t=s=0$. This concludes the proof.
\end{proof}
\end{proposition}


\section{Discrete Normal Subgroups}
We study discrete subgroups extensively because we will want to take quotients of simply connected complex almost Abelian groups in order to study connected complex almost Abelian groups.

\;

The following Lemma is adapted from Lemma 11.3 in \cite{Hall15}, which is stated and proven for the reals, but the same proof shows the result holds over $\C$ as well.
\begin{lemma} \label{DiscGrpVS}
Let $V$ be a finite-dimensional inner product space over $\C$, viewed as a group under vector addition, and let $\Gamma$ be a discrete subgroup of $V$. Then there exist $\R$-linearly independent vectors $v_1, \dots, v_k$ in $V$ such that $\Gamma$ is precisely the set of vectors of the form $\sum_{i=1}^k m_i v_i$ with each $m_i \in \Z$. 
\end{lemma}

Armed with the above Lemma, we now provide a bound on the rank of discrete normal subgroups of a simply connected almost Abelian group in terms of the data of $J(\aleph)$, which completely and uniquely determines the simply connected almost Abelian group.

\begin{proposition} \label{DiscSubgrpDim}
Every discrete normal subgroup $N \subseteq G$ of a simply connected almost Abelian group $G$ with Lie algebra $\aA(\aleph)$ is a free group of rank 
\begin{equation*}
    k \leq \dim_{\R}\big(\ker(J(\aleph))\big) + 2,
\end{equation*} 
generated by $\R$-linearly independent elements $[v_1,t_1], \dots, [v_k, t_k] \in Z(G) \subseteq G = \C^d \rtimes \C$.
\begin{proof}
It is known that any discrete normal subgroup is central. In Prop. \ref{CenterSimpConn} we proved that $Z(G) = \exp\big(Z(\aA(\aleph))\big) \times T_\aleph$. Also, recall from the proof of Prop. \ref{CenterSimpConn} that for any $[v,t],[u,s] \in G$, we may express the product as:
\begin{equation} \label{grpmult}
    [v,t][u,s] = [v + e^{tJ(\aleph)}u, t + s].
\end{equation}
Now by Prop. \ref{CenterSimpConn} if $[u,s] \in Z(G)$ then $u \in \ker(J(\aleph))$ implies $J(\aleph)u = 0$, which in turn implies $e^{tJ(\aleph)}u = u$. Thus by \eqref{grpmult}, we have that for $[v,t],[u,s] \in Z(G)$,
\begin{equation}\label{Ctrgrpmult}
    [v,t][u,s] = [v + u, t + s].
\end{equation}

Thus if we define $f: G \to \C^{d+1}$ to be the homeomorphism $f([v,t]) = (v,t)$, then $f|_{Z(G)}$ is a group homomorphism as well by \eqref{Ctrgrpmult}, and since we are working with matrix Lie groups, a Lie group homomorphism. A Lie group homomorphism maps discrete subgroups to discrete subgroups, thus $f(N)$ is a discrete subgroup.

Thus by Lemma \ref{DiscGrpVS}, $f(N)$ is a free Abelian group generated by $\R$-linearly independent elements $v_1, \dots, v_k \in \C^{d+1}$, and their span satisfies
\begin{equation*}
    \C\{v_i\}_{i=1}^k \subseteq \C\{f(Z(G))\},
\end{equation*}
which implies that
\begin{equation} \label{DiscDimIneq}
    k \leq \dim_{\R}\big( \C \{f(Z(G))\} \big).
\end{equation}

What remains to be shown is that $k \leq \dim_{\C}\big(\ker(J(\aleph))\big) + 2$, which we will show by proving
\begin{equation*}
    \dim_{\R}\big( \C \{f(Z(G))\} \big) \leq \dim_{\R}\big(\ker(J(\aleph))\big) + 2. \tag{$*$}
\end{equation*}
Recall from Prop. \ref{CenterSimpConn} that $Z(G) = \exp_{G}\big(Z(\aA(\aleph))\big) \times T_\aleph$. Now $\dim_{\R}(T_\aleph) \leq 2$ implies that $(*)$ holds if and only if 
\begin{equation} \label{DimIneqToProve}
    \dim_{\R}\big( \R \{ f(\exp_{G}[Z(\aA(\aleph))]) \} \big) \leq \dim_{\R}(\ker(J(\aleph)))
\end{equation}
holds. Let $\{w_i\}_{i=1}^m$ be a basis for $\C\{(f \circ \exp_{G})(Z(\aA(\aleph)))\}$. Without loss of generality, we may suppose $\{w_i\}_{i=1}^m \subseteq (f \circ \exp)(Z(\aA(\aleph)))$. Recall from Remark \ref{ExpIsId} that $\exp_{G}(v,t) = [v,t]$ for all $(v,t) \in \ker(J(\aleph)) \oplus \C \supseteq Z(\aA(\aleph))$. Then \eqref{DimIneqToProve} follows from the fact that $\exp_{G}((v,t)) = [v,t]$ implies $(f \circ \exp_{G})(v,t) = (v,t)$. 
\end{proof}
\end{proposition}


\section{Subgroups and Subalgebras}
We classify connected subgroups, prove the nonexistence of compact connected subgroups of a simply connected group $\widetilde{G}$, and study some relationships between subgroups of $\widetilde{G}$ and quotients $G \coloneqq \widetilde{G}/N$ of $\widetilde{G}$ by discrete normal subgroups $N$.

\begin{remark} \label{SubalgebraForms}
Let $G$ be a simply connected almost Abelian Lie group with Lie Algebra $\aA(\aleph) = \C^d \rtimes \C$. Then by Proposition 4 in  \cite{Ave16} every  Lie subalgebra $\mathbf{L} \subset \aA(\aleph)$ takes one of the following two forms: \begin{enumerate}
    \item $\mathbf{L} = \mathbf{W} \subset \C^d$ is an Abelian Lie subalgebra.
    \item $\mathbf{L}$ is of the form $$\mathbf{L} = \left\{ (w + tv_0, t) \in \C^d \rtimes \C | w \in \mathbf{W}, t \in \C \right\},$$ 
    where $v_0 \in \C^d$ is fixed and $\mathbf{W} \subset \C^d$ is an $\ad_{e_0}$-invariant vector subspace. Here $\mathbf{L}$ is Abelian if and only if $\mathbf{W} \subset Z(\aA(\aleph)).$
\end{enumerate}
\end{remark}


Recall that,\footnote{Theorem 5.20 in \cite{Hall15}} to every Lie subalgebra $\mathbf{L}$ of a (any) Lie algebra there exists a unique connected Lie subgroup $H_\mathbf{L}$ with Lie algebra $\bL$.  We now find explicit forms for these connected Lie subgroups.


\begin{proposition} \label{ConnSubGrpForms}
The connected Lie subgroup $H_{\bL} \subset G$ of the simply connected almost Abelian Lie group
$G$ with Lie algebra $\bL$ as in Remark \ref{SubalgebraForms} is given by either of the following two forms, accordingly:
\begin{enumerate}
    \item[1.] $$H_{\bL} = \left\{ [w, 0] \in \C^d \rtimes \C \;|\; w \in \mathbf{W} \right\} = \exp({\bL}),$$
    \item[2.] $$H_{\bL} = \left\{ \left[w + \frac{e^{tJ(\aleph)} - \mathbbm{1}}{J(\aleph)}v_0, t\right] \in \C^d \rtimes \C \;|\; w \in \mathbf{W},\;\; t \in \C \right\} \cong \;\; \exp(\mathbf{W}) \rtimes \C. $$
\end{enumerate}
\begin{proof}
That $H_{\bL}$ is indeed a Lie group can be checked via the faithful matrix representation in Prop. \ref{SimpConnRep} and further the product rule as given in Prop. \ref{CenterSimpConn}. To show closure under the group operation in Case 2, we observe that for any 
$$\left[w + \frac{e^{tJ(\aleph)} - \mathbbm{1}}{J(\aleph)}v_0, t\right], \left[u + \frac{e^{sJ(\aleph)} - \mathbbm{1}}{J(\aleph)}v_0, s\right] \in H_{\bL}, \quad w, u \in W,\quad t,s \in \C,$$
we have
\begin{multline*}
      \left[w + \frac{e^{tJ(\aleph)} - \mathbbm{1}}{J(\aleph)}v_0, t\right] \left[u + \frac{e^{sJ(\aleph)} - \mathbbm{1}}{J(\aleph)}v_0, s\right]  \\ 
      = \left[(w + e^{tJ(\aleph)}u) +  \frac{e^{(t+s)J(\aleph)} - \mathbbm{1}}{J(\aleph)}v_0, t+s \right].
\end{multline*}
For this to be in $H_{\bL}$ we need $e^{tJ(\aleph)}u \in \mathbf{W}$, which is guaranteed because $\mathbf{W}$ is $J(\aleph)$-invariant.

For Case 1, the exponential map as given in Lemma \ref{SimpConnExp} gives the desired result directly. For Case 2, take some $(w_0 + t_0v_0, t_0) \in \bL$. Consider a path $\gamma: (-1, 1) \longrightarrow \mathbf{W} \oplus \C$ defined by 
$$\gamma(\tau) = (w(\tau), t(\tau)),$$ where $$(w(0), t(0)) = (0,0), \qquad (w'(0), t'(0)) = (w_0, t_0) \in \mathbf{W} \oplus \C.$$ 
Then we have, $$\dv{}{\tau}\left[w + \frac{e^{t(\tau)J(\aleph)} - \mathbbm{1}}{J(\aleph)}v_0, t(\tau)\right] \biggr|_{\tau = 0} = (w_0 + t_0v_0, t_0).$$
Thus the Lie algebra of $H_{\bL}$ is $\bL$. 

Next, consider the map $\Phi: H_{\bL} \rightarrow 
\exp(\mathbf{W}) \rtimes \C$ given by $$\Phi[v,t] = \left[ -\frac{e^{tJ(\aleph)} - \mathbbm{1}}{J(\aleph)}v_0 + v, t\right].$$
It can be easily shown that $\Phi$ is a Lie group isomorphism and we provide the details for completeness. First, we check that $\Phi$ is bijective by showing its inverse is given by: 
$$\Phi^{-1}[v,t] = \left[ \frac{e^{tJ(\aleph)} - \mathbbm{1}}{J(\aleph)}v_0 + v, t\right].$$
Indeed,
\begin{equation*}
    \Phi^{-1} \circ \Phi[v,t] = \left[ \frac{e^{tJ(\aleph)} - \mathbbm{1}}{J(\aleph)}v_0 + \left[-\frac{e^{tJ(\aleph)} - \mathbbm{1}}{J(\aleph)}v_0 + v,\right] t\right] = [v, t].
\end{equation*}
We next show $\Phi$ is a Lie group homomorphism, 
\begin{align*}
    \Phi[v,t] \cdot \Phi[u,s] &=   \left[-\frac{e^{tJ(\aleph)} - \mathbbm{1}}{J(\aleph)}v_0 + v, t\right] \cdot  \left[-\frac{e^{tJ(\aleph)} - \mathbbm{1}}{J(\aleph)}v_0 + u, s\right] \\
    &= \left[-\frac{e^{(t+s) J(\aleph)} - \mathbbm{1}}{J(\aleph)}v_0 + (v + e^{tJ(\aleph)}u), t + s\right]  \\
    &= \Phi[v + e^{tJ(\aleph)}u, t+s] = \Phi([v,t] \cdot [u,s]).
\end{align*}
Lastly note that for $ \left[w + \frac{e^{tJ(\aleph)} - \mathbbm{1}}{J(\aleph)}v_0, t\right] \in H_{\bL}$, $$\Phi \left[w + \frac{e^{tJ(\aleph)} - \mathbbm{1}}{J(\aleph)}v_0, t\right] = \left[ -\frac{e^{tJ(\aleph)} - \mathbbm{1}}{J(\aleph)}v_0 + w + \frac{e^{tJ(\aleph)} - \mathbbm{1}}{J(\aleph)}v_0, t\right] = [w, t],$$
which is indeed and element of $\exp(\mathbf{W}) \rtimes \C$, showing that $\Phi$ maps $H_{\bL} \rightarrow \exp(\mathbf{W}) \rtimes \C$ as desired. 
\end{proof}
\end{proposition}


\begin{remark} \label{NoConnCompSubGrps}
Prop. \ref{ConnSubGrpForms} implies that no connected subgroup $H$ of a simply connected almost abelian group $G$ is compact.
\end{remark}


\begin{lemma} \label{ProjSubgrpCorr}
Let $\widetilde{G}$ be a simply connected almost Abelian Lie group and $N \subseteq \widetilde{G}$ a normal subgroup. Let $G \coloneqq \widetilde{G} / N$ be the resultant connected almost Abelian Lie group. Then every connected subgroup $H \subseteq G$ is the projection $H = \widetilde{H}/N$ of a unique connected Lie subgroup $\widetilde{H} \subseteq \widetilde{G}$.
\begin{proof} Since we have a simply connected almost abelian group, we may use the matrix representation given in Prop. \ref{SimpConnRep}, and thus our almost abelian group is a matrix Lie group, which is to say it is a closed subgroup of $GL_n(\C)$.

Let $\mathbf{L}_{\widetilde{G}}$ be the Lie algebra of $\widetilde{G}$, and let $\mathbf{L}_{G}$ be the Lie algebra of $G$. The quotient map $q_N : \widetilde{G} \to G$ is a surjective complex Lie group homomorphism, and its derivative $dq_N : \mathbf{L}_{\widetilde{G}} \to \mathbf{L}_{G}$ is a surjective Lie algebra homomorphism. The preimage $dq_N\inv(\mathbf{L}_{H})$ of the Lie algebra $\mathbf{L}_H$ of $H$ is a Lie subalgebra of $\mathbf{L}_{\widetilde{G}}$, and thus is the Lie algebra of the unique connected subgroup $\widetilde{H} \leq \widetilde{G}$. The image $q_N(\widetilde{H}) \leq G$ is a connected subgroup with Lie algebra $\mathbf{L}_{H}$, which by uniqueness must be $q_N(\widetilde{H}) = H$. Finally, if $\widetilde{H}' \leq \widetilde{G}$ is another connected subgroup with $q_N(\widetilde{H}') = H$ then $\mathbf{L}_{\widetilde{H}'} = \mathbf{L}_H$, so that again by uniqueness $\widetilde{H}' = \widetilde{H}$.
\end{proof}
\end{lemma}

\begin{lemma} \label{OGLem10}
 Let $G$ be a simply connected almost Abelian group, $N \subseteq G$  a discrete normal subgroup and $H \subseteq G$ a connected subgroup. Then there exists a subgroup $B \subseteq N $ such that $N = (N \cap H) \times B$. 
\begin{proof}
We use Prop. \ref{ConnSubGrpForms} to write:
\begin{equation} \label{FormsOfConnSubgrp}
    H \cong \begin{cases}
        \exp(\mathbf{W}) \\
        \exp(\mathbf{W}) \times \C \\
        \exp(\mathbf{W}) \rtimes \C
    \end{cases}
\end{equation}
where $\mathbf{W} \subseteq \C^d$ is a vector subspace.

We will prove $N \cap H$ is a pure subgroup, and use Corollary 28.3 in \cite{fuchs1970infinite} to reach our desired conclusion that $N \cap H$ is a direct factor. Let $[v,t] \in N$ and $n \in \N \st [v,t]^n = [nq, nt] \in N \cap H$. Then $nv \in \mathbf{W}$ implies $v \in \mathbf{W}$. Furthermore, if we are in the first case in \eqref{FormsOfConnSubgrp}, then $t = 0$ implies $nt = 0$. Meanwhile, if we are in the second or third case case in \eqref{FormsOfConnSubgrp}, we would have $t \in \C$ implies $\frac{t}{n} \in \C$. In either case we therefore have $[v,t] \in H$. Since $[v,t] \in N$ by assumption, we therefore have $[v,t] \in N \cap H$. By Corollary 28.3 in \cite{fuchs1970infinite}, we therefore have that $N \cap H$ is a direct factor of $N$.
\end{proof}
\end{lemma}


\begin{proposition} \label{NonCompactConnAAGrp}
Let $G \coloneqq \widetilde{G} / \Gamma$ be a connected almost Abelian Lie group (where $\widetilde{G}$ is the simply connected universal cover). Then $G$ is never compact. 
\begin{proof}
Let $\aA(\aleph)$ be the Lie algebra of $G$ and $\widetilde{G}$, and let $\aA(\aleph) = \C^{d} \oplus \C$ (where $\C^d$ is an Abelian subalgebra). Recall from \cite{Ave16} that for any almost Abelian Lie algebra
\begin{equation*}
    \ker(\ad_{e_0}) = Z(\aA(\aleph)),
\end{equation*}
and so in particular
\begin{equation*}
    \dim_{\C}(\ker(\text{ad}_{e_0})) = \dim_{\C}(\ker(J(\aleph))) = \dim_{\C}(Z(\aA(\aleph))).
\end{equation*}
Assume for contradiction $\dim_{\C}(\ker(J(\aleph))) = d$ (note that we in general cannot have that $\dim_{\C}(\ker(J(\aleph))) = d+1$ as the equality would force algebra to be Abelian). First, notice that if $\mathbf{V}$ is the orthogonal space to the dimension $d$ central subspace $\ker(J(\aleph)) = Z(\aA(\aleph))$ of $\aA(\aleph)$, then $\mathbf{V}$ is Abelian because it has dimension 1. Now any element $W \in \aA(\aleph)$ can be linearly decomposed: $W = W_1 + W_2$, such that $W_1 \in \mathbf{V}$ and $W_2 \in Z(\aA(\aleph))$. By linearity of the Lie bracket then, we have that for all $\alpha V \in \mathbf{V},\ [\alpha V, W] = \alpha[V, W_1] + \alpha[V, W_2] = 0 + 0 = 0$. Thus $\mathbf{V} \subseteq Z(\aA(\aleph))$ which implies $Z(\aA(\aleph)) = \aA(\aleph)$, and thus $\aA(\aleph)$ is Abelian, a contradiction (note, $\mathbf{V} \subseteq Z(\aA(\aleph))$ was already a contradiction).

So there exists $X \in \C^d$ such that $X \notin \ker(\text{ad}_{e_0}) = Z(\aA(\aleph))$. Consider the one parameter subgroup $H_X = \{ \exp_{\widetilde{G}}(\tau X) \, | \, \tau \in \C\}$. By Prop. \ref{DiscSubgrpDim}, we know that $Z(\widetilde{G}) = \exp_{\widetilde{G}}(Z(\aA(\aleph))) \times T_\aleph$. By construction, $\exp_{\widetilde{G}}(X) \notin \exp_{\widetilde{G}}(Z(\aA(\aleph)))$.

Now assume for contradiction there exists $\tau \in \C^\times$ such that $\exp_{\widetilde{G}}(\tau X)$ is an element of $\exp_{\widetilde{G}}(Z(\aA(\aleph)))$. Then $\tau X \in Z(\aA(\aleph))$ implies $[\tau X, Y] = 0 \foreach Y \in \aA(\aleph)$, which implies $[X,Y] = 0 \foreach Y \in \aA(\aleph)$ and thus $X \in Z(\aA(\aleph))$, a contradiction. Thus $\exp_{\widetilde{G}}(\tau X) \notin \exp_{\widetilde{G}}(Z(\aA(\aleph))) \foreach \tau \in \C^\times$. Thus $H_X \cap Z(G) = \{1\}$, and thus $H_X \cap \Gamma = \{1\}$ because all discrete normal subgroups of a Lie group are central. 

Now consider the quotient map 
\begin{equation*}
    \pi|_{H_X}: H_X \to \widetilde{G}/\Gamma
\end{equation*}
which---when not restricted---is also a surjective complex Lie group homomorphism.

Now due to our result that $H_X$ intersects $\Gamma$ only at the identity of $\widetilde{G}$, we have that $H_X \cong H_X / (H_X \cap \Gamma)$. Simultaneously, we know that $\ker(\pi|_{H_X}) = H_X \cap \Gamma$. Thus, $\pi(H_X) \cong H_X$. 

Consider the algebra representation for $\aA(\aleph)$ as:
\begin{equation*}
    \aA(\aleph) = \left\{ \begin{pmatrix}
    0 & 0 & 0 \\
    v & {tJ(\aleph)} & 0 \\
    t & 0 & 0
    \end{pmatrix} \middle| (v,t) \in \C^d \times \C \right\}
\end{equation*}
From Prop. \ref{SimpConnRep}, we know that the matrix exponential takes this representation of the Lie algebra to the simply connected Lie group that has $\aA(\aleph)$ as its Lie algebra. By Lemma \ref{SimpConnExp}, we see that, specifically, the exponential of an element of $\aA(\aleph)$ under this representation is:
\begin{equation} \label{SimpConnExpRep}
    \exp_{\widetilde{G}}\begin{pmatrix}
        0 & 0 & 0 \\
        v & {tJ(\aleph)} & 0 \\
        t & 0 & 0
    \end{pmatrix}
    =
    \begin{pmatrix}
        1 & 0 & 0 \\
        \frac{e^{t J(\aleph)} - \mathbbm{1}}{tJ(\aleph)}v  & e^{tJ(\aleph)} & 0 \\
        t & 0 & 1
    \end{pmatrix}
\end{equation}
Now recall we defined $X$ to be an element of $\C^d$. So plugging in $(\tau X,0)$ for $(v,t)$ into \eqref{SimpConnExpRep}, we get
\begin{equation*}
    \exp_{\widetilde{G}}\begin{pmatrix}
        0 & 0 & 0 \\
        \tau X & 0 & 0 \\
        0 & 0 & 0
    \end{pmatrix}
    =
    \begin{pmatrix}
        1 & 0 & 0 \\
        \tau X  & 1 & 0 \\
        0 & 0 & 1
    \end{pmatrix}
\end{equation*}
From this, it is apparent that $H_X$ is closed (contains all its limit points) and path connected (and thus connected).

Since any connected subgroup of $\widetilde{G}$ is not compact by Remark \ref{NoConnCompSubGrps}, $H_X$ is not compact. Then the image of $H_X$ under $\pi$ is not compact, and therefore $G$ contains a noncompact Lie subgroup. If $G$ were to be compact, then any closed subgroup would be compact as well. However, we showed that $H_X$ and thus\footnote{Since a quotient map is an open map.} $\pi(H_X)$ is a closed non-compact Lie subgroup, thus $G$ cannot be compact.
\end{proof}
\end{proposition}


Prop. \ref{NonCompactConnAAGrp} establishes that connected almost Abelian Lie groups cannot be compact, but it says nothing about the compactness of subgroups of connected almost Abelian Lie groups. Our next result gives a necessary and sufficient condition for the compactness of a subgroup of a connected almost Abelian Lie group.

\begin{proposition} \label{CompCondition}
Let $G$ be a connected almost Abelian group, and $H \subseteq G$ a connected Lie subgroup. Let $\widetilde{H} \subseteq \widetilde{G}$ be the connected Lie  group such that $H = \widetilde{H} / \Gamma$ (where $\Gamma$ is a discrete normal subgroup), and let $\widetilde{G}$ be the simply connected almost Abelian group such that $G = \widetilde{G} / \Gamma$. Then $\rank{(\Gamma \cap \widetilde{H})} = \dim_{\R}(\widetilde{H}) = \dim_{\R}(H)$ if and only if $H$ is compact.
\begin{proof}
First we note that since $\widetilde{H}$ is the universal covering group, the equality $\dim_{\R}(H) = \dim_{\R}(\widetilde{H})$ holds regardless.

By Lemma \ref{OGLem10}, we know there exists $B \subseteq \Gamma$ such that $\Gamma = (\Gamma \cap \widetilde{H}) \times B$. Define $\Gamma' \coloneqq \Gamma \cap \widetilde{H}$. Since $B \cap \widetilde{H} = \{0\}$ by construction, we may write:
\begin{equation*}
    H = \widetilde{H} / \Gamma = \widetilde{H} / (\Gamma' \times B) = \widetilde{H} / \Gamma'
\end{equation*}
Thus we have obtained a discrete subgroup $\Gamma'$ contained in $\widetilde{H}$ yielding the same (isomorphic) quotient as when viewing $\widetilde{H}$ as a subgroup of $\widetilde{G}$ and taking the quotient by $\Gamma$.

Since $H = \widetilde{H} / \Gamma'$, $H$ is connected. By Prop. 4 in \cite{Ave16}, all subalgebras of an almost Abelian Lie algebra are either Abelian or almost abelian. It is known that there is a 1-1 correspondence between connected subgroups and Lie subalgebras, implying that any underlying group is either almost Abelian (by definition) or is Abelian (having an Abelian Lie algebra implies the connected component of the identity is abelian for real Lie groups, and complex Lie groups are in particular real Lie groups). We analyze the two cases:

\noindent
\textbf{Case 1:} First we will consider the case where $H$ is almost Abelian.

For the forwards implication, assume for contradiction that $\rank{(\Gamma')} = \dim_{\R}(\widetilde{H}) = \dim_{\R}(H)$ (i.e., we will show the claim is vacuous in this case). Now $\Gamma' \subseteq Z(\widetilde{H}) \cong \C^\ell \cong \R^{2\ell}$ for some $\ell \in \N$, $\ell < d$. Let $k = 2 \ell + 2 = \dim_{\R}(\widetilde{H})$. Thus we have that there exists a minimal generating set $\{[v_1, t_1], \dots, [v_k, t_k]\}$. 

By Lemma \ref{TalephTrivCond}, we have that either $T_{\aleph} \cong \{0\}$ or $T_{\aleph} \cong \Z$. In either case, $\rank{T_{\aleph}} \leq 1$. And thus in either case, there exists some $[u_1, t_0] \in \Gamma'$ such with minimal (in magnitude) $t_0 \in \C$ such that $t_j = n_j t_0,\ 1 \leq j \leq k,\ \{n_j\}_{j=1}^k \subseteq \N$. Then by Lemma 7 in \cite{AAB20}, or equivalently by row operations, we have that there is a generating set $\{[u_1, t_0], [u_2, 0], \dots, [u_k, 0] \} \subseteq \Gamma'$. By presupposition, $\{[v_j, t_j\}_{j=1}^{k}$ was a minimal generating set for $\Gamma'$. Now note that $\abs{\{u_j\}_{j=2}^k} = 2 \ell + 1 > 2 \ell$, while also $\{[u_j, 0]\}_{j=2}^k \subseteq \R^{2\ell}$. Thus there must be at least two $u_j$'s that are $\R$-linearly dependent. But then we may find a minimal generating set of $\Gamma'$ of cardinality $< k$, a contradiction to $\rank{\Gamma'} = \dim_{\R}(\widetilde{H}) = k$. Thus it is impossible for $\rank(\Gamma') = \dim_{\R}(\widetilde{H})$, and thus this direction of implication holds vacuously.

Now for the reverse implication. Since $H$ is almost Abelian, we have that $H$ is a connected almost Abelian Lie group, and therefore by Prop. \ref{NonCompactConnAAGrp}, $H$ cannot be compact, and thus the reverse implication is vacuous as well, and so we are done with this case.
\newline

\noindent
\textbf{Case 2:} We proceed to check the Abelian case. Note that $H$ is Abelian if and only if $\widetilde{H}$ is Abelian because they have the same Lie algebra. If $\widetilde{H}$ is Abelian, then $\widetilde{H} \cong \C^n$ by Prop. 5 in \ref{ConnSubGrpForms}. By assumption, $\Gamma'$ is a discrete subgroup, so it is generated by $k \leq 2n$ $\R$-linearly independent elements. 
Therefore there exists an isomorphism $\varphi: \Gamma' \to \Z^k$. So we have:
\begin{equation} \label{HCongTo}
    H \cong \widetilde{H} / \Gamma' \cong  \C^n / \varphi(\Gamma') \cong \C^n / \Z^k \cong \T^{\left(2\floor{\frac{k}{2}} \right)} \times (\C / \Z )^{\epsilon} \times \C^\eta  
\end{equation}
where
\begin{equation*}
    \epsilon = \begin{cases}
        0 & k \equiv 0 \mod 2 \\
        1 & k \equiv 1 \mod 2 
    \end{cases} \quad \quad
    \eta = \begin{cases}
    n - \frac{k}{2} & k \equiv 0 \mod 2 \\
    n - \floor{\frac{k}{2}}  - 1 & k \equiv 1 \mod 2.
    \end{cases}
\end{equation*}

For the forwards direction, assume $\rank{(\Gamma')} = \dim(\widetilde{H})$. Then by \eqref{HCongTo}, we would have $H \cong \T^{\floor{\frac{k}{2}}} = \T^n$, which is compact.

Now for the reverse implication, assume instead that $H$ is compact. Then it is apparent from \eqref{HCongTo} that the only way for this to occur is for $\epsilon = \eta = 0$ and so $n - \frac{k}{2} = 0$, implying $k = 2n$ and $\rank{(\Gamma')} = \dim_{\R}(\widetilde{H})$, as desired.
\end{proof}
\end{proposition}

While the necessary and sufficient condition described by Prop. \ref{CompCondition} appears unwieldy, the condition can be a powerful technical tool in the proofs of more concrete results, such as Prop. \ref{MaxCompSubgrp} later on.
\section{Discrete Subgroups}
To study discrete subgroups, it is useful to study their images under a certain map which we will call the \textit{projection homomorphism}. In the next lemma, we give the definition and check that the map is indeed a well-defined homomorphism.

\begin{lemma} \label{ProjIsHom}
Let $G$ be an $n$-dimensional simply connected almost Abelian Lie group with Lie algebra $\aA(\aleph)$. Recalling that $G = \C^d \rtimes \C \eqqcolon N \rtimes H$, define $P:G\to\mathbb{C}$ by
\begin{equation*}
    P([v,t]) \coloneqq t.
\end{equation*}
$P$ is a group homomorphism.

\begin{proof}
Observe that
\begin{equation*}
    [v,0][0,t] = [v + e^0 0, t] = [v,t].
\end{equation*}
Thus, we may represent $[v,t]$ as $nh$ for some $n \in N$ and $h \in H$. Note that by the definition of the semidirect product, $N \unlhd \widetilde{G}$. Then for $n_1, n_2 \in N$ and $h_1, h_2 \in H$, we have
\begin{align*}
    P[n_1 h_1 n_2 h_2] &= P[n_1 (h_1 n_2 h_1\inv) h_1 h_2] \\
    &= P[n_3 h_1 h_2] \\
    &= \pi_2(h_1 h_2) \\
    &= P[n_1 h_1] P[n_2 h_2].
\end{align*}
where $\pi_2$ is the projection onto the second factor.
\end{proof}
\end{lemma}

The utility of this projection homomorphism is shown in the next result. In particular, we show that the finite generation of $D$ is equivalent to that of its image under the projection homomorphism.

\begin{proposition} \label{FinGenCond}
Let $G = \C^d \rtimes \C \eqqcolon N \rtimes H$ be a simply connected almost Abelian group, and $D \subseteq G$ a discrete subgroup. Then $D$ is finitely generated if and only if $P(D) \subseteq H$ is finitely generated.
\begin{proof} We identify $N$ with an internal subgroup of $G$.

In one direction, if $D$ is finitely generated, then so is $P(D)$ because $P$ is a homomorphism by Lemma \ref{ProjIsHom}. 

Conversely, suppose $P(D) \subseteq H$ is finitely generated. So there exists $k \in \N$ and $\{\alpha_i\}_{i=1}^k \subseteq D$ such that $\Z\{P(\alpha_i) : 1 \leq i \leq k\} = P(D)$. Let $x \in D$ be arbitrary. Then we know $P(x) = P(\alpha)$ for some $\alpha \in \Z\{\alpha_i\}_{i=1}^k$. Then $x\alpha\inv \in \ker(P) = N$. Therefore $x\alpha\inv \in D \cap N$ so, 
\begin{equation} \label{DecompOfD}
  D = (D \cap N) (\Z\{\alpha_i\}_{i=1}^k)
\end{equation}
Now $N = \C^d \cong \R^{2d}$ hence $D \cap N$ is an additive subgroup of $\R^{2d}$, so it is finitely generated. Thus by \eqref{DecompOfD} we know $D$ is also finitely generated.
\end{proof}
\end{proposition}

When the projection of a discrete subgroup is discrete, it is finitely generated, so the discrete subgroup itself itself must be finitely generated by Prop. \ref{FinGenCond}. The following lemma tells us what can be ascertained if the projection of a discrete subgroup fails to be discrete.
\begin{lemma} \label{NcapDInZ}
Let $G = N \rtimes H = \C^d \rtimes \C$ be a simply connected almost Abelian group, and let $D \subseteq G$ be a discrete subgroup. If $P(D)$ is not discrete, then $N \cap D \subseteq (Z(G))_0$.

\begin{proof} We identify $N$ and $H$ with internal subgroups of $G$.

Note that $[v,t][u,s] = [v + e^{tJ(\aleph)}u, t+s]$. Thus, if we can show that $J(\aleph) = 0$, we will have shown $N \cap D \subseteq (Z(\widetilde{G}))$.

Recall that
\begin{equation} \label{adAsDer}
    \ad_X = \dv{}{\tau} \Ad_{e^{\tau X}} \biggr|_{\tau = 0}.
\end{equation}
In \eqref{adAsDer}, let $X = e_0$ (where $H = \exp_{G}(\C e_0)$). So we have
\begin{equation} \label{ade_0AsDerV1}
    \ad_{e_0} = \dv{}{\tau} \Ad_{e^{\tau e_0}} \biggr|_{\tau = 0}.
\end{equation}
Also recall from Remark \ref{ExpIsId} that $\exp_{G}|_{\ker(J(\aleph)) \times \C}((v,t)) = [v,t]$. Thus we may rewrite \eqref{ade_0AsDerV1} as:
\begin{equation} \label{ade_0asDerV2}
    \ad_{e_0} = \dv{}{\tau} \Ad_{\tau} \biggr|_{\tau = 0}.
\end{equation}
using the fact that the matrix exponential and the exponential associated with our Lie group coincide when we use the Lie algebra representation associated with $G$, as in Prop. \ref{SimpConnRep}.

We note that
\begin{equation} \label{ExpOf(v,0)}
    \exp_{G}((v,0)) = I + \sum_{n=1}^\infty \frac{1}{n!}
    \begin{pmatrix}
        0 & 0 & 0 \\
        v & 0 & 0 \\
        0 & 0 & 0
    \end{pmatrix}^n
    = 
    \begin{pmatrix}
        1 & 0 & 0 \\
        v & 1 & 0 \\
        0 & 0 & 1
    \end{pmatrix}
    = [v,0].
\end{equation}
Thus for all $[n,0] \in N$, we have that there exists $(n,0) \in \text{Lie}(\widetilde{G}) \eqqcolon \mathfrak{g}$ such that $\exp_{\widetilde{G}}((n,0)) = [n,0]$. So we may identify $D \cap N$ with a subset $\mathfrak{d} \coloneqq \{(d,0) : [d,0] \in D \cap N \} \subseteq \mathfrak{g}$. Note that $\mathfrak{d}$ is also discrete.

Recall that the continuous action of a connected group on a discrete set is trivial. Now, $H$ is a connected group, and we can consider the action of conjugation by elements of $H$ on the discrete set $\mathfrak{d}$. Denote this action by $C: H \times \mathfrak{d} \to \mathfrak{d}$. Conjugation is a continuous action, thus the action $C$ must be trivial. That is, $C_h \equiv \id_{\mathfrak{d}}$ for every $h \in H$. In particular, the action $C |_{P(D) \times \mathfrak{d}}$ is trivial.

Because we are working in a Hausdorff space, the limit in \eqref{ade_0asDerV2} can be computed using any sequence $(t_n)_{n=1}^{\infty}$ with $t_n\to0$. Assume $P(D)$ is not discrete. Then there is a sequence of $t_n \in P(D)$ such that $t_n \to 0$. Now $\Ad_{\tau}$ is just conjugation by $\tau \in H$, which we know to be trivial when acting on $\mathfrak{d}$. Since there exists our sequence $(t_n)_{n=1}^\infty$ that is contained in $P(D)$ by construction, we have that
\begin{equation*}
     \ad_{e_0} |_{\mathfrak{d}} = \dv{}{\tau} \Ad_{\tau} \big|_{\mathfrak{d}}  \biggr|_{\tau = 0} = \lim_{t_n \to 0} \frac{\Ad_{t_n}|_{\mathfrak{d}} - \Ad_{0}|_{\mathfrak{d}}}{t_n - 0} = \lim_{t_n \to 0} \frac{I - I}{t_n} = 0.
\end{equation*}
Because $\ad_{e_0}|_{\mathfrak{d}} = J(\aleph)|_{\mathfrak{d}}$, we have that $J(\aleph)|_{\mathfrak{d}} = 0$. Since $[v,t][u,s] = [v + e^{tJ(\aleph)}u, t + s] \foreach [v,t],[u,s] \in G$, we can conlude that $D \cap N \subseteq Z(G)$.

Now that we know $\ad_{e_0} |_{\mathfrak{d}} = 0$, we calculate that for all $(d,0) \in \mathfrak{d}$ and for arbitrary $(v,t) \in \mathfrak{g}$,
\begin{align*}
    [(d,0), (v,t)] &= [(d,0), (v,0)] + [(d,0), (0,t)] \\
        &= 0 - t \ad_{e_0}((d,0)) \\
        &= 0.
\end{align*}
Thus $\mathfrak{d} \subseteq Z(\aA(\aleph))$. So\footnote{E.g. by exercise 9.1 in \cite{harris1991representation}}, we have that $\exp_{G}(\mathfrak{d}) \subseteq Z(G)_0$. As $\exp_{G}(\mathfrak{d}) = D \cap N$ by construction, we conclude $D \cap N \subseteq Z(G)_0$.
\end{proof}
\end{lemma}


\begin{remark} \label{discreteKerSet}
For a finitely supported multiplicity function $\aleph$, we define
\begin{equation*}
    \mathfrak{k} \coloneqq \left\{t \in \C \; \middle| \; \ker\left( \frac{e^{tJ(\aleph)} - \mathbbm{1}}{tJ(\aleph)}\right) \neq \{0\} \right\}.
\end{equation*}
Then,
\begin{equation} \label{SecondSet}
\mathfrak{k}=\left\{ \frac{2\pi  i  m}{x_p} \middle| m \in \Z,\ p \in \supp{\aleph} \cap (\C - \{0\}) \right\}.
\end{equation}
\end{remark}

\begin{proof} We abbreviate the set on the right hand side of \ref{SecondSet} as $S$.

In one direction, let $t \in \C$ such that $\ker\left( \frac{e^{tJ(\aleph)} - \mathbbm{1}}{tJ(\aleph)}\right) \neq \{0\}$. By rank-nullity, $ \det [\frac{e^{tJ(p,n)} - \mathbbm{1}}{tJ(p,n)}] = 0$. Now, because $[\frac{e^{tJ(p,n)} - \mathbbm{1}}{tJ(p,n)}]$ is an upper triangular matrix, $ \det [\frac{e^{tJ(p,n)} - \mathbbm{1}}{tJ(p,n)}] =  \left( \frac{e^{tx_p}-1}{tx_p}\right)^n$, so $ \left( \frac{e^{tx_p}-1}{tx_p}\right)^n = 0$. Simplifying, $ \left( \frac{e^{tx_p}-1}{tx_p}\right) = 0$, $e^{tx_p}-1 = 0$. If $t=0$, we can think of $\frac{e^{tx_p}-1}{tx_p}$ as $ \lim_{z\to 0} \frac{e^z-1}{z} $ which equals $1$, so  $ \det [\frac{e^{tJ(p,n)} - \mathbbm{1}}{tJ(p,n)}] \neq 0$. This is a contradiction, so $t \neq 0$. Thus, since $t\neq 0$ we get that $e^{tx_p} - 1 = 0 $, so $e^{tx_p} = 1$. Thus $tx_p \in 2 \pi  i  \Z$ and as $tx_p \neq 0$, $x_p \neq 0$ so we can divide and obtain $t \in \frac{2 \pi  i  \Z}{x_p}$. Hence, $\mathfrak{k}\subseteq S$.

In the other direction, let $t \in S$. Then $tx_p \in 2 \pi  i  \Z$ so $e^{tx_p} = 1$ and thus  $e^{tx_p} -1 = 0$. This means $\frac{e^{tx_p}-1}{tx_p} = 0$ so $\det(J(p,n)) = \left( \frac{e^{tx_p}-1}{tx_p}\right)^n = 0$. 
So by rank-nullity,
\begin{equation*}
    \ker\left( \frac{e^{tJ(p,n)} - \mathbbm{1}}{tJ(p,n)} \right) \neq \{0\}.
\end{equation*}
Note that 
\begin{equation*}
\mathfrak{k}=\bigcup_{p \in \supp{\aleph}} \bigcup_{n=1}^\infty \left\{t \in \C \middle| \ker\left( \frac{e^{tJ(p,n)} - \mathbbm{1}}{tJ(p,n)}\right) \neq \{0\} \right\}
\end{equation*}
Therefore, $t \in \mathfrak{k}$. Hence, $S\subseteq\mathfrak{k}$.
\end{proof}

We now come to one of our main results---that every discrete subgroup of a simply connected almost Abelian group is finitely generated. This refines Lemma \ref{NcapDInZ} and tells us that even when the projection of a discrete subgroup fails to be discrete, it is still at least finitely generated.

\begin{theorem} \label{AllDiscSubgrpsOfSimpConnGrpAreFinGen}
Let $G$ be a simply connected almost Abelian group. Every discrete subgroup $D \subseteq G$ is finitely generated.
\begin{proof}
If $P(D) \subseteq \C$ is discrete then it is finitely generated, and so by Proposition \ref{FinGenCond}, $D$ is finitely generated and we are done. So assume $P(D)$ is not discrete.
\newline

\noindent
\textbf{Case 1:} Suppose $P(D) \subseteq \mathfrak{k}$. Consider the group
\begin{equation*}
    H \coloneqq \left\langle \left\{ \frac{2 \pi  i }{x_p} \, \middle| \, p \in \supp{\aleph} \cap (\C - \{0\}) \right\} \right\rangle.
\end{equation*}
Since $J(\aleph)$ is a finite multiplicity function, $\abs{\left\{ x_p \, : \, p \in \supp{\aleph} \cap (\C - \{0\}) \right\}} \in \N$. Thus $H$ is finitely generated. Note that $H$ is Abelian. By assumption and construction:
\begin{equation*}
    P(D) \subseteq \mathfrak{k} \subseteq H.
\end{equation*}
Since all subgroups of a finitely generated Abelian group are finitely generated, it follows that $P(D)$ is finitely generated. By Prop. \ref{FinGenCond}, we are done.
\newline

\noindent
\textbf{Case 2:} Suppose $P(D) \not \subseteq \mathfrak{k}$. Then we have that there exists $t_0 \in P(D) \cap \mathfrak{k}^c$. Since $t_0 \in \mathfrak{k}^c$, we have by definition that
\begin{equation} \label{kerIsTriv}
    \ker\left( \frac{e^{t_0 J(\aleph)} - \mathbbm{1}}{t_0 J(\aleph)}\right) = \{0\}.
\end{equation}
Now if $v \in \ker(J(\aleph))$, then of course $v \in \ker(e^{t_0 J(\aleph)} - \mathbbm{1})$. For the reverse inclusion, assume that $v \in \ker(e^{t_0 J(\aleph)} - \mathbbm{1})$ and that $t_0 \in \mathfrak{k}^c$. Then, recalling $t_0 \neq 0$ by construction,
\begin{equation*}
    \left(\sum_{k=1}^\infty \frac{1}{n!} t_0^n J(\aleph)^{n-1} \right)v = 0,
\end{equation*}
if and only if
\begin{equation*}
    J(\aleph) \left(\sum_{k=1}^\infty \frac{1}{n!} t_0^n J(\aleph)^{n-1} \right)v = 0,
\end{equation*}
if and only if
\begin{equation*}
    t_0\left(\sum_{k=1}^\infty \frac{1}{n!} t_0^{n-1} J(\aleph)^{n-1} \right)J(\aleph)v = 0,
\end{equation*}
and thus $v \in \ker(J(\aleph))$. Thus $[v_0, t_0]$ also satisfies
\begin{equation} \label{kerEquality}
    \ker\left( e^{t_0 J(\aleph)} - \mathbbm{1} \right) = \ker(J(\aleph)).
\end{equation}

Because of \eqref{kerIsTriv} and the rank-nullity theorem, we have that $\frac{e^{t_0 J(\aleph)} - \mathbbm{1}}{J(\aleph)}$ is invertible, and we denote the inverse by $\frac{J(\aleph)}{e^{t_0 J(\aleph)} - \mathbbm{1}}$. Hence, $\gamma=\frac{J(\aleph)}{e^{t_0 J(\aleph)} - \mathbbm{1}}v_0\in\mathbb{C}^d$ is well-defined.

As shown in the proof of Prop. 5., we can select $\Phi\in\Aut{(G)}$ given by
\[\Phi([v_0,t_0])=\left[v_0 - \frac{e^{t_oJ(\aleph)} - \mathbbm{1}}{J(\aleph)} \gamma, t_0 \right] = [0, t_0].\]

By considering $\Phi(D)$ instead of $D$ we can assume without loss of generality that $[0, t_0] \in D$. Using the formula for matrix multiplication like that in Prop. \ref{CenterSimpConn}, we get:
\begin{equation*}
    [v,t][u,s][v,t]\inv [u,s]\inv = \left[ \left(e^{tJ(\aleph)} - \mathbbm{1} \right)u - \left(e^{sJ(\aleph)} - \mathbbm{1} \right)v, 0 \right],
\end{equation*}
for all $[v,t],[u,s] \in G$. Denote the commutator of $G$ by $[\cdot,\cdot]_{\widetilde{G}}$. It follows in particular that $[G,G] \subseteq \C^d \cap D$. 

Consider the map $\varphi_{[0,t_0]}: D \to \C^d \cap D$ given by:
\begin{equation} \label{commutatorByt_0Map}
    \varphi_{[0,t_0]}([u,s]) = [0,t_0][u,s][0,t_0]\inv [u,s]\inv = \left[ \left(e^{t_0} J(\aleph) - \mathbbm{1} \right)u, 0 \right],
\end{equation}
for all $[u,s] \in D$. Since $P(D) \cong D / \C^d$ is assumed to be non-discrete, by Lemma \ref{NcapDInZ} and the earlier realization that $[G,G]_{\widetilde{G}} \subseteq \C^d \cap G$, we have that $[D,D]_{D} \subseteq \C^d \cap D \subseteq Z(G)_0$. So in particular, $\varphi_{[0,t_0]}([u,s])$ commutes with all elements of $G$. Hence we may calculate:
\begin{align*}
    \varphi_{[0,t_0]}([v,t][u,s]) &= [0,t_0] [v,t] [u,s] [0,t_0] [0,t_0]\inv ([v,t][u,s])\inv \\
    &= [0,t_0] [v,t] [u,s] [0,t_0]\inv [u,s]\inv [v,t]\inv \\
    &= [0,t_0][v,t][0,t_0]\inv \left( [0,t_0] [u,s] [0,t_0]\inv [u,s]\inv \right) [v,t]\inv \\
    &= [0,t_0][v,t][0,t_0]\inv \varphi_{[0,t_0]}([u,s]) [v,t]\inv \\
    &= \varphi_{[0,t_0]}([v,t]) \varphi_{[0,t_0]}([u,s]),
\end{align*}
and thus $\varphi_{[0,t_0]}$ is a homomorphism. From \eqref{kerEquality} we have that 
\begin{equation*}
    \ker(\varphi) = \{[u,s] \in D \divides u \in \ker(J(\aleph))\} = (Z(G)_0 \times \C) \cap D.
\end{equation*}
From Prop. \ref{CenterSimpConn}, we know that $Z(G)_0 \times \C \cong \C^{1 + \dim(Z(G))}$, and thus $(Z(G)_0 \times \C) \cap D$ is finitely generated. Since both $\ker(\varphi_{[0,t_0]})$ and $\varphi_{[0,t_0]}(D)$ are finitely generated, the result follows from the same logic as in the proof of Prop. \ref{FinGenCond}.
\end{proof}
\end{theorem}

\section{Homogeneous Spaces} \label{SectionHomogeneousSpaces}
In this section, we describe a characterization of the maximal compact subgroup of a connected almost Abelian group. Such a characterization is of interest, since the homotopy type of a Lie group is given by that of its maximal compact subgroup. We begin with a few technical lemmas.

\begin{lemma}[Covering Space of a Homogeneous Space]
Let $G$ be a simply connected almost Abelian Lie group, $H \subseteq G$ be a closed subgroup with $H_0$ as its identity component. Then $G / H_0$ is the universal cover of $G / H$.
\begin{proof}
By Prop. 1.94(b) in \cite{knapp2013lie}, the natural map of $G/H_0$ onto $G/H$ is a covering map. By Prop. 1.94(e) in \cite{knapp2013lie}, $G/H$ is simply connected.
\end{proof}
\end{lemma}


\begin{lemma} \label{intersectComplexSubgrps}
The intersection of complex connected Lie subgroups of a simply connected complex almost Abelian group is again a complex connected Lie subgroup.
\begin{proof}

Let $H$ and $H'$ be connected subgroups of a complex almost Abelian group $G$. By Prop. \ref{ConnSubGrpForms}, we have that $H$ and $H'$ are one of the two forms:
\begin{enumerate}
    \item[(i)] $$H = \left\{ [w, 0] \in \C^d \rtimes \C \;|\; w \in \mathbf{W} \right\},$$
    \item[(ii)] $$H = \left\{ \left[w + \frac{e^{tJ(\aleph)} - \mathbbm{1}}{J(\aleph)}v_0, t\right] \in \C^d \rtimes \C \;|\; w \in \mathbf{W},\;\; t \in \C \right\},$$
\end{enumerate}
where $\mathbf{W}$ is an $\ad_{e_0}$-invariant subspace, and $v_0 \in \C^d$ is an arbitrary fixed element. Thus the intersection $H \cap H'$ can be broken into three cases.
\newline

\noindent
\textbf{Case 1:} Assume $H$ and $H'$ are of type (i). Then their intersection is the intersection of subspaces of $\C^d$, and thus is a subspace, which is of course connected.
\newline

\noindent
\textbf{Case 2:} Assume $H$ is of type (i) and $H'$ is of type (ii). Thus we may say
\begin{align*}
    H &\coloneqq \left\{ [w,0] \; \middle| \; w \in \mathbf{W} \right\}, \\
    H' & \coloneqq \left\{ \left[w' + \frac{e^{tJ(\aleph)} - \mathbbm{1}}{J(\aleph)}v_0', t \right] \in \C^d \rtimes \C \; \middle| \; w' \in \mathbf{W'}, t \in \C \right\}.
\end{align*}
But then any element $[w,t]$ of $H \cap H'$ is in particular an element of $H$, and so $t = 0$. So the elements of $H'$ with $t = 0$ are of the form
\begin{equation*}
    \left[ w' + \frac{e^{0 \cdot J(\aleph)} - \mathbbm{1}}{J(\aleph)}v_0', 0 \right] = [w', 0],
\end{equation*}
and thus $H \cap H'$ is another subspace, and so again is connected.
\newline

\noindent
\textbf{Case 3:} Assume $H$ and $H'$ are both of form (ii). So we define them as:
\begin{align*}
    H & \coloneqq \left\{ \left[w + \frac{e^{tJ(\aleph)} - \mathbbm{1}}{J(\aleph)}v_0, t \right] \in \C^d \rtimes \C \; \middle| \; w \in \mathbf{W}, t \in \C \right\}, \\
    H' & \coloneqq \left\{ \left[w' + \frac{e^{tJ(\aleph)} - \mathbbm{1}}{J(\aleph)}v_0', t \right] \in \C^d \rtimes \C \; \middle| \; w' \in \mathbf{W'}, t \in \C \right\},
\end{align*}
where $v_0,v_0' \in \C^d$ are arbitrary fixed elements, and $\mathbf{W}, \mathbf{W'}$ are $\ad_{e_0}$-invariant subspaces. Note that if the intersection $H \cap H'$ is empty then we are done, so we may assume the intersection is nonempty.

Observe that for $w\in\mathbf{W}$, we have that $\left[ w + \frac{e^{tJ(\aleph)} - \mathbbm{1}}{J(\aleph)}v_0, t \right] \in H \cap H'$ if and only if there exists $w'\in\mathbf{W}'$ such that $w + \frac{e^{tJ(\aleph)} - \mathbbm{1}}{J(\aleph)}v_0 = w' + \frac{e^{tJ(\aleph)} - \mathbbm{1}}{J(\aleph)}v_0'$ or $w = w' +  \frac{e^{tJ(\aleph)} - \mathbbm{1}}{J(\aleph)}(v_0' - v_0)$. Thus the first component of $H \cap H'$ consists of elements $w +  \frac{e^{tJ(\aleph)} - \mathbbm{1}}{J(\aleph)}v_0$ such that $w \in \mathbf{W} \cap \left(\mathbf{W'} +  \frac{e^{tJ(\aleph)} - \mathbbm{1}}{J(\aleph)}(v_0' - v_0) \right)$.
We have two subcases: either $v_0 = v_0'$, or $v_0 \neq v_0'$. 
\newline

\noindent
\textbf{Subcase 1:} Suppose $v_0 = v_0'$. Then, $[w,t]\in H\cap H'$ implies, as above,
\[w \in \mathbf{W} \cap \left(\mathbf{W'} +  \frac{e^{tJ(\aleph)} - \mathbbm{1}}{J(\aleph)}(v_0' - v_0) \right)=w \in \mathbf{W} \cap \left(\mathbf{W'} +  \frac{e^{tJ(\aleph)} - \mathbbm{1}}{J(\aleph)}(\vec{0}) \right)=W\cap W'.\]
As in Case 1, we have an intersection of subspaces which we know to be connected.
\newline

\noindent
\textbf{Subcase 2:} Suppose $v_0 \neq v_0'$. We prove that $H \cap H'$ is path connected, and so in particular connected. 

In order to prove path-connectedness, it is sufficient to prove that there exists a path between any two of the affine spaces
\begin{equation*}
    \mathbb{A}_t \coloneqq \left\{ \left[ w + \frac{e^{tJ(\aleph)} - \mathbbm{1}}{J(\aleph)}v_0, t \right] \; \middle| \; w \in \mathbf{W} \st \exists w' \in \mathbf{W'} \st w' + \frac{e^{tJ(\aleph)} - \mathbbm{1}}{J(\aleph)}v_0' = w + \frac{e^{tJ(\aleph)} - \mathbbm{1}}{J(\aleph)}v_0 \right\}
\end{equation*}

Now consider $(H \cap H')_0$. Being a connected subgroup of $G$, we have that it is once again of the form (i) or (ii). If it is in the form (ii), we have that there are elements $[*,t] \in (H \cap H')_0$ for $\abs{\R}$ distinct nonzero $t$. 

If it is instead of the form $(i)$ there are two more options: either there are no $t \neq 0$ coordinates, or there exists $g \coloneqq \left[ w + \frac{e^{rJ(\aleph)} - \mathbbm{1}}{J(\aleph)}v_0, r \right] \in H \cap H'$ such that $g \notin (H \cap H')_0$. However, since in particular $g \in H$, there exists a neighborhood $B_{\delta_1}(g) \subseteq H$ in the subspace topology on $H$. Similarly, there exists a neighborhood $B_{\delta_2}(g) \subseteq H'$. Now by the assumptions on $H$, we have that $\left[ w + \frac{e^{sJ(\aleph)} - \mathbbm{1}}{J(\aleph)}v_0, s \right] \in H$ for all $s \in \C$. Since the function $\frac{e^{sJ(\aleph)} - \mathbbm{1}}{J(\aleph)}$ is continuous in $s$, and the last component function is obviously continuous, we have that there is a neighborhood $B_{\epsilon_1}(r) \subseteq \C$ such that $\left[ w + \frac{e^{sJ(\aleph)} - \mathbbm{1}}{J(\aleph)}v_0, s \right] \in B_{\delta_1}(g)$ for all $s \in B_{\epsilon_1}(r)$. It can be easily seen that an analogous statement holds for $H'$, with corresponding neighborhood $B_{\epsilon_2}(r)$. Let $\delta \coloneqq \min(\delta_1, \delta_2)$, and let $\epsilon = \min(\epsilon_1, \epsilon_2)$. Then it is apparent that $B_\delta(g) \subseteq H \cap H'$ contains elements $[*, s]$ for all $s \in B_{\epsilon}(r)$. Thus in this case as well there are $\abs{\R}$ elements with distinct $t$-coordinates contained in $H \cap H'$. 

Suppose that $[w,t]\in H\cap H'$ implies that $t=0$. Then, we have
\[w\in\mathbf{W}\cap\left(\mathbf{W}'+\frac{e^{tJ(\aleph)} - \mathbbm{1}}{J(\aleph)}(v_0' - v_0) \right)=\mathbf{W}\cap\left(\mathbf{W}'+[0](v_0' - v_0) \right)=\mathbf{W}\cap\mathbf{W}',\]
so this subcase reduces to Case 1 where we had intersecting subspaces of $\mathbb{C}^d$, which is clearly connected.

Assume $\mathbb{A}_t$ has $\abs{\R}$ elements with distinct $t$-values, which we showed must be the case if $t$ is not always zero for all elements of $H \cap H'$. Then there are uncountably many points $[*,t] \in H \cap H'$, while by Remark \ref{discreteKerSet} there can only be a countable number of points $[*,t]$ such that $\frac{e^{tJ(\aleph)} - \mathbbm{1}}{J(\aleph)}$ is not invertible. So choose some $h \coloneqq \left[ w_0 + \frac{e^{t_0 J(\aleph)} - \mathbbm{1}}{J(\aleph)}v_0,\ t_0 \right] \in H \cap H'$ such that $\frac{e^{t_0 J(\aleph)} - \mathbbm{1}}{J(\aleph)}$ is invertible. Note that since $h \in H \cap H'$, we know that there exists
$w_0' \in \mathbf{W'}$ such that $w_0 = w_0' + \frac{e^{t_0 J(\aleph)} - \mathbbm{1}}{J(\aleph)}(v_0' - v_0)$. Then consider the path
\begin{equation*}
    \gamma(s) \coloneqq \left[ \left( \frac{e^{sJ(\aleph)} - \mathbbm{1}}{J(\aleph)} \right) \left( \frac{e^{t_0 J(\aleph)} - \mathbbm{1}}{J(\aleph)} \right)^{-1} \left(w_0' + \frac{e^{t_0 J(\aleph)} - \mathbbm{1}}{J(\aleph)} (v_0' - v_0) \right)  + \frac{e^{s J(\aleph)} - \mathbbm{1}}{J(\aleph)}v_0,\ s \right].
\end{equation*}
Clearly, $\gamma$ is a continuous function. Now we note,
\begin{multline*}
    \left( \frac{e^{sJ(\aleph)} - \mathbbm{1}}{J(\aleph)} \right) \left( \frac{e^{t_0 J(\aleph)} - \mathbbm{1}}{J(\aleph)} \right)^{-1} \left(w_0' + \frac{e^{t_0 J(\aleph)} - \mathbbm{1}}{J(\aleph)} (v_0' - v_0) \right) = \\ \underbrace{\left( \frac{e^{sJ(\aleph)} - \mathbbm{1}}{J(\aleph)} \right) \left( \frac{e^{t_0 J(\aleph)} - \mathbbm{1}}{J(\aleph)} \right)^{-1} w_0'}_{\in \mathbf{W'} \text{ by $\ad_{e_0}$-invariance of } \mathbf{W'}} + \underbrace{\left( \frac{e^{sJ(\aleph)} - \mathbbm{1}}{J(\aleph)} \right) (v_0' - v_0)}_{\in \frac{e^{sJ(\aleph)} - \mathbbm{1}}{J(\aleph)}(v_0' - v_0) }.
\end{multline*}
So
\begin{equation*}
    \left( \frac{e^{sJ(\aleph)} - \mathbbm{1}}{J(\aleph)} \right) \left( \frac{e^{t_0 J(\aleph)} - \mathbbm{1}}{J(\aleph)} \right)^{-1} \left(w_0' + \frac{e^{t_0 J(\aleph)} - \mathbbm{1}}{J(\aleph)} (v_0' - v_0) \right) \in \mathbf{W'} +  \frac{e^{sJ(\aleph)} - \mathbbm{1}}{J(\aleph)}(v_0' - v_0).
\end{equation*}
Simultaneously, we have that 
\begin{equation*}
    \left( \frac{e^{sJ(\aleph)} - \mathbbm{1}}{J(\aleph)} \right) \left( \frac{e^{t_0 J(\aleph)} - \mathbbm{1}}{J(\aleph)} \right)^{-1} w \in \mathbf{W},
\end{equation*}
by the $\ad_{e_0}$-invariance of $\mathbf{W}$. Thus the image of $\gamma$ is in $H\cap H'$.
\end{proof}
\end{lemma}


\begin{definition}
Let $X$ be a subset of a complex almost Abelian Lie group $G$. We define $\mathcal{C}(X)$ to be the minimal connected complex Lie subgroup containing $X$ (defined as the intersection of all such connected complex Lie groups). Note that by Lemma \ref{intersectComplexSubgrps}, this is well-defined.
\end{definition}

The main result is that the maximal compact subgroup of an almost Abelian Lie group $G=\widetilde{G}/\Gamma$ is intimately related to $\mathcal{C}(\Gamma)$.

\begin{proposition} \label{MaxCompSubgrp}
Let $G = \widetilde{G}/\Gamma$ be a connected almost Abelian Lie group. The maximal compact subgroup $K\subseteq G$ is given by $K = \mathcal{C}(\Gamma) / \Gamma$.
\end{proposition}
\begin{proof}
Recall that all subgroups of $G$ are either Abelian or almost Abelian. If a subgroup of $G$ is almost Abelian, by Prop. \ref{NonCompactConnAAGrp}, it cannot be compact. Hence, all compact subgroups of $G$ are Abelian and $\Lie{K}$ is an Abelian subalgebra of $\Lie{G}$.

We claim that $\Lie{K}\subseteq\mathbb{C}\{\log{\Gamma}\}$. Suppose to the contrary that this is not true. Then, there exists $X\in\Lie{K}-\mathbb{C}\{\log{\Gamma}\}$. Since $X\notin\mathbb{C}\{\log{\Gamma}\}$, there is no $\tau\in\mathbb{C}^*$ and $\gamma\in\Gamma-\{\mathbbm{1}\}$ such that $X=\frac{1}{\tau}\log{\gamma}$. That is, we cannot write $\gamma=e^{\tau X}$ for any choice of $\gamma\in\Gamma-\{\mathbbm{1}\}$ and $\tau\in\mathbb{C}^*$. It follows that $\left(H_X-\{\mathbbm{1}\}\right)\cap\left(\Gamma-\{\mathbbm{1}\}\right)=\emptyset$. Since subgroups intersect at least at the identity element, we must have $H_X\cap\Gamma=\{\mathbbm{1}\}$.

Now consider the quotient map $q_{\Gamma}\colon\widetilde{G}\to G$. Since this is a covering map, it is continuous and open. By the definition of quotient topology, $S\subseteq G$ is open if and only if $q_{\Gamma}^{-1}(S)\subseteq\widetilde{G}$ is open. Since $q_{\Gamma}^{-1}(S^C)=q_{\Gamma}^{-1}(S)^C$, we immediately have that $S\subseteq G$ is closed if and only if $q_{\Gamma}^{-1}(S)\subseteq\widetilde{G}$ is closed. Now observe that $q_{\Gamma}^{-1}(q_{\Gamma}(S))=S\cdot\Gamma$. If $S$ is closed, then since $\Gamma$ is also closed, we have that $S\cdot\Gamma$ is closed. But by our observation above, $S\cdot\Gamma=q_{\Gamma}^{-1}(q_{\Gamma}(S))$ is closed if and only if $q_{\Gamma}(S)$ is closed. Consequently, $q_{\Gamma}$ is also a closed map.

Since $H_X$ is closed, $q_{\Gamma}(H_X)\subseteq G$ is also closed. However, by construction, $q_{\Gamma}(H_X)\subseteq K$ is not compact since $q_{\Gamma}$ is continuous and $H_X$ is not compact. Since closed subsets of compact spaces are themselves compact, $K$ is not compact, a contradiction. Hence, $\Lie{K}\subseteq\mathbb{C}\{\log{\Gamma}\}$.

By Prop. \ref{CompCondition}, the compactness of $K$ implies that $\dim{K} = \dim{\widetilde{K}} = \rank{(\Gamma\cap\widetilde{K})} \leq \rank{\Gamma}$, where $K = \widetilde{K}/\Gamma$. Observe that if $\widetilde{K} = \mathcal{C}(\Gamma)$, we have by construction $\widetilde{K} \cap \Gamma = \Gamma$, thus $\dim{K}$ obtains the upper bound $\rank{\Gamma}$. Moreover, since the Lie algebra of $K$ is in the complex span of the logarithm of $\Gamma$, we conclude that we must have $\widetilde{K} = \mathcal{C}(\Gamma)$.
\end{proof}

Hence, we have a construction of the maximal compact subgroup of a connected complex almost Abelian Lie group. The relative simplicity of this construction suggests that it may be much easier to probe the homotopy type of such Lie groups by instead studying the homotopy type of their maximal compact subgroups.



\section*{Acknowledgments}  This work was done as part of the University of California, Santa Barbara Mathematics Summer Research Program for Undergraduates and was supported by NSF REU Grant DMS 1850663. We are very grateful to both UCSB and the NSF for making this opportunity possible, and for the enriching, challenging, and fun experiences we had in the course of the program. 

\printbibliography


\end{document}